\documentclass{amsart}
\usepackage{amssymb}
\usepackage{amsbsy, amsthm, amsmath,amstext, amsopn}
\usepackage[all]{xy}
\usepackage{amsfonts}
\usepackage{amscd}
\usepackage{stmaryrd}
\usepackage{color}

\newtheorem{thm}{Theorem} [section]
\newtheorem{lemma}[thm]{Lemma}

\newtheorem{corollary}[thm]{Corollary}
\newtheorem{prop}[thm]{Proposition}

\newtheorem{setting}[thm]{Setting}
\theoremstyle{definition}

\newtheorem{defn}[thm]{Definition}

\newtheorem{example}[thm]{Example}

\theoremstyle{remark}

\newtheorem{remark}[thm]{Remark}

\begin{document}

\numberwithin{equation}{section}

\newcommand{\hs}{\mbox{\hspace{.4em}}}
\newcommand{\ds}{\displaystyle}
\newcommand{\bd}{\begin{displaymath}}
\newcommand{\ed}{\end{displaymath}}
\newcommand{\bcd}{\begin{CD}}
\newcommand{\ecd}{\end{CD}}

\newcommand{\on}{\operatorname}
\newcommand{\proj}{\operatorname{Proj}}
\newcommand{\bproj}{\underline{\operatorname{Proj}}}

\newcommand{\spec}{\operatorname{Spec}}
\newcommand{\Spec}{\operatorname{Spec}}
\newcommand{\bspec}{\underline{\operatorname{Spec}}}
\newcommand{\pline}{{\mathbf P} ^1}
\newcommand{\aline}{{\mathbf A} ^1}
\newcommand{\pplane}{{\mathbf P}^2}
\newcommand{\aplane}{{\mathbf A}^2}
\newcommand{\coker}{{\operatorname{coker}}}
\newcommand{\ldb}{[[}
\newcommand{\rdb}{]]}

\newcommand{\Sym}{\operatorname{Sym}^{\bullet}}
\newcommand{\Symp}{\operatorname{Sym}}
\newcommand{\Pic}{\bf{Pic}}
\newcommand{\Aut}{\operatorname{Aut}}
\newcommand{\PAut}{\operatorname{PAut}}

\newcommand{\too}{\twoheadrightarrow}
\newcommand{\C}{{\mathbf C}}
\newcommand{\Z}{{\mathbf Z}}
\newcommand{\Q}{{\mathbf Q}}
\newcommand{\Cx}{{\mathbf C}^{\times}}
\newcommand{\Cbar}{\overline{\C}}
\newcommand{\Cxbar}{\overline{\Cx}}
\newcommand{\cA}{{\mathcal A}}
\newcommand{\cS}{{\mathcal S}}
\newcommand{\cZ}{{\mathcal Z}}
\newcommand{\cV}{{\mathcal V}}
\newcommand{\cM}{{\mathcal M}}
\newcommand{\bA}{{\mathbf A}}
\newcommand{\cB}{{\mathcal B}}
\newcommand{\cC}{{\mathcal C}}
\newcommand{\cD}{{\mathcal D}}
\newcommand{\D}{{\mathcal D}}
\newcommand{\cs}{{\mathbf C} ^*}
\newcommand{\boldc}{{\mathbf C}}
\newcommand{\cE}{{\mathcal E}}
\newcommand{\cF}{{\mathcal F}}
\newcommand{\bF}{{\mathbf F}}
\newcommand{\cG}{{\mathcal G}}
\newcommand{\G}{{\mathbb G}}
\newcommand{\cH}{{\mathcal H}}
\newcommand{\CI}{{\mathcal I}}
\newcommand{\cJ}{{\mathcal J}}
\newcommand{\cK}{{\mathcal K}}
\newcommand{\cL}{{\mathcal L}}
\newcommand{\baL}{{\overline{\mathcal L}}}

\newcommand{\Mf}{{\mathfrak M}}
\newcommand{\bM}{{\mathbf M}}
\newcommand{\bm}{{\mathbf m}}
\newcommand{\cN}{{\mathcal N}}
\newcommand{\theo}{\mathcal{O}}
\newcommand{\cP}{{\mathcal P}}
\newcommand{\cR}{{\mathcal R}}
\newcommand{\Pp}{{\mathbb P}}
\newcommand{\boldp}{{\mathbf P}}
\newcommand{\boldq}{{\mathbf Q}}
\newcommand{\bbL}{{\mathbf L}}
\newcommand{\cQ}{{\mathcal Q}}
\newcommand{\cO}{{\mathcal O}}
\newcommand{\Oo}{{\mathcal O}}
\newcommand{\cY}{{\mathcal Y}}
\newcommand{\OX}{{\Oo_X}}
\newcommand{\OY}{{\Oo_Y}}
\newcommand{\otY}{{\underset{\OY}{\ot}}}
\newcommand{\otX}{{\underset{\OX}{\ot}}}
\newcommand{\cU}{{\mathcal U}}\newcommand{\cX}{{\mathcal X}}
\newcommand{\cW}{{\mathcal W}}
\newcommand{\boldz}{{\mathbf Z}}
\newcommand{\qgr}{\operatorname{q-gr}}
\newcommand{\gr}{\operatorname{gr}}
\newcommand{\rk}{\operatorname{rk}}
\newcommand{\Sh}{\operatorname{Sh}}
\newcommand{\SH}{{\underline{\operatorname{Sh}}}}
\newcommand{\End}{\operatorname{End}}
\newcommand{\uEnd}{\underline{\operatorname{End}}}
\newcommand{\Hom}{\operatorname{Hom}}
\newcommand{\uHom}{\underline{\operatorname{Hom}}}
\newcommand{\uHomY}{\uHom_{\OY}}
\newcommand{\uHomX}{\uHom_{\OX}}
\newcommand{\Ext}{\operatorname{Ext}}
\newcommand{\bExt}{\operatorname{\bf{Ext}}}
\newcommand{\Tor}{\operatorname{Tor}}

\newcommand{\inv}{^{-1}}
\newcommand{\airtilde}{\widetilde{\hspace{.5em}}}
\newcommand{\airhat}{\widehat{\hspace{.5em}}}
\newcommand{\nt}{^{\circ}}
\newcommand{\del}{\partial}

\newcommand{\supp}{\operatorname{supp}}
\newcommand{\GK}{\operatorname{GK-dim}}
\newcommand{\hd}{\operatorname{hd}}
\newcommand{\id}{\operatorname{id}}
\newcommand{\res}{\operatorname{res}}
\newcommand{\lrar}{\leadsto}
\newcommand{\im}{\operatorname{Im}}
\newcommand{\HH}{\operatorname{H}}
\newcommand{\TF}{\operatorname{TF}}
\newcommand{\Bun}{\operatorname{Bun}}

\newcommand{\F}{\mathcal{F}}
\newcommand{\Ff}{\mathbb{F}}
\newcommand{\nthord}{^{(n)}}
\newcommand{\Gr}{{\mathfrak{Gr}}}

\newcommand{\Fr}{\operatorname{Fr}}
\newcommand{\GL}{\operatorname{GL}}
\newcommand{\gl}{\mathfrak{gl}}
\newcommand{\SL}{\operatorname{SL}}
\newcommand{\ff}{\footnote}
\newcommand{\ot}{\otimes}
\def\Ext{\operatorname {Ext}}
\def\Hom{\operatorname {Hom}}
\def\Ind{\operatorname {Ind}}
\def\bbZ{{\mathbb Z}}

\newcommand{\nc}{\newcommand}
\nc{\ol}{\overline} \nc{\cont}{\on{cont}} \nc{\rmod}{\on{mod}}
\nc{\Mtil}{\widetilde{M}} \nc{\wb}{\overline} \nc{\wt}{\widetilde}
\nc{\wh}{\widehat} \nc{\sm}{\setminus} \nc{\mc}{\mathcal}
\nc{\mbb}{\mathbb}  \nc{\K}{{\mc K}} \nc{\Kx}{{\mc K}^{\times}}
\nc{\Ox}{{\mc O}^{\times}} \nc{\unit}{{\bf \on{unit}}}
\nc{\boxt}{\boxtimes} \nc{\xarr}{\stackrel{\rightarrow}{x}}

\nc{\Ga}{\G_a}
 \nc{\PGL}{{\on{PGL}}}
 \nc{\PU}{{\on{PU}}}

\nc{\h}{{\mathfrak h}} \nc{\kk}{{\mathfrak k}}
 \nc{\Gm}{\G_m}
\nc{\Gabar}{\wb{\G}_a} \nc{\Gmbar}{\wb{\G}_m} \nc{\Gv}{G^\vee}
\nc{\Tv}{T^\vee} \nc{\Bv}{B^\vee} \nc{\g}{{\mathfrak g}}
\nc{\gv}{{\mathfrak g}^\vee} \nc{\BRGv}{\on{Rep}\Gv}
\nc{\BRTv}{\on{Rep}T^\vee}
 \nc{\Flv}{{\mathcal B}^\vee}
 \nc{\TFlv}{T^*\Flv}
 \nc{\Fl}{{\mathfrak Fl}}
\nc{\BRR}{{\mathcal R}} \nc{\Nv}{{\mathcal{N}}^\vee}
\nc{\St}{{\mathcal St}} \nc{\ST}{{\underline{\mathcal St}}}
\nc{\Hec}{{\bf{\mathcal H}}} \nc{\Hecblock}{{\bf{\mathcal
H_{\alpha,\beta}}}} \nc{\dualHec}{{\bf{\mathcal H^\vee}}}
\nc{\dualHecblock}{{\bf{\mathcal H^\vee_{\alpha,\beta}}}}
\newcommand{\ramBun}{{\bf{Bun}}}
\newcommand{\ramBuno}{\ramBun^{\circ}}

\nc{\Buntheta}{{\bf Bun}_{\theta}} \nc{\Bunthetao}{{\bf
Bun}_{\theta}^{\circ}} \nc{\BunGR}{{\bf Bun}_{G_\BR}}
\nc{\BunGRo}{{\bf Bun}_{G_\BR}^{\circ}}
\nc{\HC}{{\mathcal{HC}}}
\nc{\risom}{\stackrel{\sim}{\to}} \nc{\Hv}{{H^\vee}}
\nc{\bS}{{\mathbf S}}
\def\BRep{\operatorname {Rep}}
\def\Conn{\operatorname {Conn}}

\nc{\Vect}{{\operatorname{Vect}}}
\nc{\Hecke}{{\operatorname{Hecke}}}

\newcommand{\ZZ}{{Z_{\bullet}}}
\nc{\HZ}{{\mc H}\ZZ} \nc{\eps}{\epsilon}

\nc{\CN}{\mathcal N} \nc{\BA}{\mathbb A}

\nc{\ul}{\underline}

\nc{\bn}{\mathbf n} \nc{\Sets}{{\on{Sets}}} \nc{\Top}{{\on{Top}}}
\nc{\IntHom}{{\mathcal Hom}}

\nc{\Simp}{{\mathbf \Delta}} \nc{\Simpop}{{\mathbf\Delta^\circ}}

\nc{\Cyc}{{\mathbf \Lambda}} \nc{\Cycop}{{\mathbf\Lambda^\circ}}

\nc{\Mon}{{\mathbf \Lambda^{mon}}}
\nc{\Monop}{{(\mathbf\Lambda^{mon})\circ}}

\nc{\Aff}{{\on{Aff}}} \nc{\Sch}{{\on{Sch}}}

\nc{\bul}{\bullet}
\nc{\module}{{\operatorname{-mod}}}

\nc{\dstack}{{\mathcal D}}

\nc{\BL}{{\mathbb L}}

\nc{\BD}{{\mathbb D}}

\nc{\BR}{{\mathbb R}}

\nc{\BT}{{\mathbb T}}

\nc{\SCA}{{\mc{SCA}}}
\nc{\DGA}{{\mc DGA}}

\nc{\DSt}{{DSt}}

\nc{\lotimes}{{\otimes}^{\mathbf L}}

\nc{\bs}{\backslash}

\nc{\Lhat}{\widehat{\mc L}}

\newcommand{\Coh}{{\on{Coh}}}

\nc{\QCoh}{QC}
\nc{\QC}{QC}
\nc{\Perf}{\rm{Perf}}
\nc{\Cat}{{\on{Cat}}}
\nc{\dgCat}{{\on{dgCat}}}
\nc{\bLa}{{\mathbf \Lambda}}

\nc{\BRHom}{\mathbf{R}\hspace{-0.15em}\on{Hom}}
\nc{\BREnd}{\mathbf{R}\hspace{-0.15em}\on{End}}
\nc{\colim}{\on{colim}}
\nc{\oo}{\infty}
\nc{\Mod}{{\on{Mod}}}

\nc\fh{\mathfrak h}
\nc\al{\alpha}
\nc\la{\alpha}
\nc\BGB{B\bs G/B}
\nc\QCb{QC^\flat}
\nc\qc{\on{QC}}
\nc\Ups{\Upsilon}

\nc{\fg}{\mathfrak g}

\nc{\Map}{\on{Map}} 

\nc{\fX}{\mathfrak X}

\nc{\XYX}{X\times_Y X}

\nc{\ch}{\check}
%\nc{\fg}{\mathfrak g}
\nc{\fb}{\mathfrak b} \nc{\fu}{\mathfrak u} \nc{\st}{{st}}
\nc{\fU}{\mathfrak U}
\nc{\fZ}{\mathfrak Z}

\nc\fk{\mathfrak k} \nc\fp{\mathfrak p}

\nc{\BRP}{\mathbf{RP}} \nc{\rigid}{\text{rigid}}
\nc{\glob}{\text{glob}}

\nc{\cI}{\mathcal I}

\nc{\La}{\mathcal L}

\nc{\quot}{/\hspace{-.25em}/}

\nc\aff{\it{aff}}
\nc\BS{\mathbb S}

\nc\Loc{{\mc Loc}}
\nc\Tr{{\on{Tr}}}
\nc\Ch{{\mc Ch}}

\nc\ftr{{\mathfrak {tr}}}
\nc\fM{\mathfrak M}

\nc\Id{\operatorname{Id}}

\nc\bimod{\on{-bimod}}

\nc\ev{\operatorname{ev}}
\nc\coev{\operatorname{coev}}

\nc\pair{\operatorname{pair}}
\nc\kernel{\operatorname{kernel}}

\nc\Alg{\operatorname{Alg}}

\nc\init{\emptyset_{\text{\em init}}}
\nc\term{\emptyset_{\text{\em term}}}

\nc\Ev{\on{Ev}}
\nc\Coev{\on{Coev}}

\nc\es{\emptyset}
\nc\m{\text{\it min}}
\nc\M{\text{\it max}}
\nc\cross{\text{\it cr}}
\nc\tr{\on{tr}}
%\nc\dim{\on{dim}}
\nc\perf{\on{-perf}}
\nc\inthom{\mathcal Hom}
\nc\intend{\mathcal End}
\nc{\QSt}{{\mathcal QSt}}
\nc{\OO}{{\mathbb O}}

\nc{\uS}{\underline{\cS}}
\nc{\uCorr}{\underline{Corr}}
\nc{\udg}{\underline{dgCat}}

\title[Nonlinear Traces]{Nonlinear Traces}

\author{David Ben-Zvi}
\address{Department of Mathematics\\University of Texas\\Austin, TX 78712-0257}
\email{benzvi@math.utexas.edu}
\author{David Nadler}
\address{Department of Mathematics\\University
  of California\\Berkeley, CA 94720-3840}
\email{nadler@math.berkeley.edu}

\maketitle

\begin{abstract}
We combine the theory of traces in homotopical algebra with sheaf theory
in derived algebraic geometry to deduce general
fixed point and character formulas. The
formalism of dimension (or Hochschild homology) of a dualizable object
in the context of higher algebra provides a unifying
framework for classical notions such as Euler characteristics,
Chern characters, and characters of group representations. 
Moreover, the simple functoriality properties of dimensions clarify celebrated identities and extend them to new contexts.

We observe that it is advantageous to calculate dimensions, traces and their functoriality
directly in the nonlinear geometric setting of correspondence categories, where they are directly identified
with (derived versions of) loop spaces, fixed point loci and loop maps, respectively.
This results in universal nonlinear versions of Grothendieck-Riemann-Roch theorems, Atiyah-Bott-Lefschetz trace formulas, and 
Frobenius-Weyl character formulas. We can then linearize by applying sheaf theories, such as the theories of ind-coherent sheaves and $\D$-modules constructed by Gaitsgory-Rozenblyum~\cite{GR}. This recovers the familiar classical identities, in families and without any smoothness or transversality assumptions.
On the other hand, the formalism also applies to higher categorical settings not captured within a linear framework, 
such as characters of group actions on categories.
\end{abstract}

 \tableofcontents

%%%%%%%%%%%%%%%%%%%%%%%%%%%%%%%%%%%%%%%%%%%
%%%%%%%%%%%%%%%%%%%%%%%%%%%%%%%%%%%%%%%%%%%
%%%%%%%%%%%%%%%%%%%%%%%%%%%%%%%%%%%%%%%%%%%

\section{Introduction}
This paper is devoted to traces and characters in homotopical algebra
and their application to algebraic geometry and representation theory. We observe that 
many geometric fixed point and trace formulas can be expressed as linearizations
of fundamental nonlinear identities, describing dimensions and traces directly in the setting of 
correspondence categories of varieties or stacks. This gives a simple uniform perspective on (and useful
generalizations of) 
geometric character and fixed point formulas of Grothendieck-Riemann-Roch and Atiyah-Bott-Lefschetz type. In addition, 
one can also specialize the universal geometric formulas  to higher categorical settings not captured within a linear framework, 
such as characters of group actions on categories.

The paper is organized as follows: after a brief summary in Section~\ref{summary}, we give a detailed overview in Section~\ref{overview}, in three sections: first, 
the abstract functoriality of traces in higher category theory; second,
their calculation in correspondence categories in derived algebraic geometry; and third, their specialization via sheaf theories. The rest of the paper follows the same structure with more details provided.
We emphasize the formal nature and appealing simplicity of the constructions in any sufficiently derived setting. 
For example, in the second part, we work within derived algebraic geometry, but the statements and proofs should hold in any setting 
(for example, derived manifolds) with a suitable notion of fiber product to handle non-transversal intersections. 
The main objects appearing in trace formulas are the derived loop space 
(the self-intersection of the diagonal in its role as the nonlinear trace of the identity map)
and more general derived fixed point loci.
The importance of a derived setting also appears prominently in the third part, where the sheaf theories we apply must have good functorial properties with respect to fiber products. 
As a result, the theory of characters in Hochschild and cyclic homology
is expressed directly by the geometry, resulting in simpler formulations.
For example, the Todd genus in Grothendieck-Riemann-Roch and the denominators 
in the classical Atiyah-Bott formula arise naturally from derived calculations.
 
\subsection{Summary}\label{summary} We now describe the main theorem, extending classical trace and dimension formulas to a very general setting in derived algebraic geometry (including equivariance for arbitrary Lie algebroids or affine algebraic groups) without any smoothness or transversality assumptions, while
 emphasizing that the main contribution of the paper is the simple geometric formalism underlying these formulas.
For our general and formal nonlinear results, we need not assume anything about what classes of derived stacks and morphisms we work with. For applications, we need to be in a setting in which the powerful mechanism of sheaf theory is fully developed~\cite{indcoh,finiteness,GR}.

\begin{setting}\label{stack conventions} Throughout this paper, we work over a field $k$ of characteristic zero. A {\em stack} $X$ connotes either 
\begin{enumerate} 
\item a QCA derived stack in the sense of~\cite{finiteness}. In other words $X$ is quasicompact with affine diagonal and the underived inertia of $X$ is finite presentation over the underlying underived stack $X_{cl}$. 
\item an ind-inf-scheme, in the sense of~\cite{GR}. These include (derived) schemes of finite (or locally almost of finite) type and ind-schemes built out of unions of the former along closed embeddings, as well as their quotients by arbitrary Lie algebroids, or equivalently formal groupoids. 
\end{enumerate}
By a {\em proper} map we indicate a proper and schematic map, while for ind-proper indicates ind-proper and ind-inf-schematic.
All appearances of proper maps in this paper may be replaced by ind-proper ones. 
\end{setting}

Thus the class of spaces we consider includes all $k$-derived schemes of finite type and their quotients by either Lie algebroids or finite type affine group schemes. 

Given a derived stack $\pi_X:X\to \Spec k$, we denote by $\pi_{\cL X}:\cL X = \Map(S^1, X) \to \Spec k$ 
its derived loop space. In general, the derived loop space is a derived thickening of the inertia stack.
For a map $f:X\to Y$, we will  denote by 
$\cL f:\cL X\to \cL Y$ the induced map on loops.

\begin{example}\label{intro ex}
For many applications, the following two special cases are noteworthy.

When $X$ is a smooth scheme, $\cL X\simeq  \BT_X[-1] $ is the total space of the shifted tangent space by the HKR theorem. The same holds for an arbitrary scheme, if we replace the tangent space by the tangent {\em complex}, see for example~\cite{conns}.
For  
$f:X\to Y$ a map of schemes, $\cL f :\BT_X[-1]\to \BT_Y[-1]$ is (the shift of) the usual tangent map.

When $Y =BG$ is a classifying stack, $\cL Y\simeq G/G$ is the adjoint quotient. 
For $X$ a $G$-scheme, and $f:X/G\to BG$ the corresponding classifying map,
$\cL f:\cL (X/G)\to \cL BG \simeq G/G$ is the universal family of derived fixed point loci.
More precisely, for
any element $g\in G$,  the derived fixed point locus $X^g \subset X$ is precisely the derived fiber
$
X^g \simeq \cL(X/G) \times_{G/G} \{g\}
$
%
% with centralizer $Z(g) \subset G$ with classifying stack $B(Z(g))$.
%Let $\OO(g) \simeq B(Z(g))\subset G/G$ be  the corresponding conjugacy class, and $\mathbb X^g = \OO(g) \times_{G/G} \cL(X/G)$ the corresponding derived fiber of $\cL f$.
%For $X^g\subset X$  the derived fixed point locus of $g$, there is a canonical identification 
%$
%\mathbb X^g \simeq X^g/Z(g)
%$
%of derived stacks over $B(Z(g))$.
\end{example}

%%%%%%%%%%%%%%%%%%%
%%%%%%%%%%%%%%%%%%%
%%%%%%%%%%%%%%%%%%%
%%%%%%%%%%%%%%%%%%%

%%%%%%%%%%%%%%%%%%%
%%%%%%%%%%%%%%%%%%%
%%%%%%%%%%%%%%%%%%%
%%%%%%%%%%%%%%%%%%%
We will measure stacks $X$ by differential graded (dg) enhancements of derived categories of sheaves.

The most familiar is the assignment $X\mapsto\cQ(X)$ of the (unbounded) category of quasicoherent sheaves. However,
we will make essential use of Grothendieck-Serre duality, in the guise of an adjunction $(f_*,f^!)$ between push-forward and extraordinary pullback for proper maps $X$. This duality is most naturally expressed in the setting of {\em ind-coherent sheaves} $X\mapsto \cQ^!(X)$ as developed in~\cite{indcoh,GR}. Ind-coherent sheaves agree with quasicoherent sheaves for smooth schemes but differ on singular schemes, where [bounded complexes of] coherent sheaves (the compact objects of $\cQ^!$) differ from perfect complexes (the compact objects of $\cQ$). In other words, ind-coherent sheaves are to coherent sheaves and $G$-theory (the setting of Grothendieck-Riemann-Roch theorems) as quasicoherent sheaves are to perfect complexes and $K$-theory. 

Another sheaf theory to which the general formalism developed in~\cite{GR} applies is the theory of $\cD$-modules $X\to \cD(X)$, which for smooth schemes agrees with the classical notion of quasicoherent complexes of modules for the sheaf of differential operators $\cD_X$ (i.e. with compact objects given by bounded coherent complexes of $\cD_X$-modules). In general~\cite{crystals,GR} $\cD(X)$ is defined as the category of crystals, i.e., as ind-coherent complexes on the de Rham space of $X$
$$\cD(X):=\cQ^!(X_{dR}).$$ The compact objects in $\cD(X)$ for $X$ a scheme are the coherent $\cD$-modules, while for $X$ a stack they form a smaller class, the {\em safe} $\cD$-modules of~\cite{finiteness}.

The book~\cite{GR} develops the theories $\cQ^!$ and $\cD$ in particular as functors out of 2-categories of correspondences of schemes and stacks, with 1-morphisms from $X$ to $Y$ given by correspondences representable over $Y$ and 2-morphisms given by ind-proper ind-schematic morphisms of correspondences. This theory encodes a huge amount of structure, including in particular pullback and pushforward functors $f^!$ and $f_*$ satisfying base change, as well as the $(f_*,f^!)$ adjunction for proper (or even ind-proper) maps. They also establish symmetric monoidal properties of the sheaf theory.\footnote{We will also need from~\cite{finiteness} the construction of continuous pushforward functors for all maps of QCA stacks (such as the non-representable projection $\pi_X:X\to pt$ from a stack) and their base-change property.}  (See Sections~\ref{sheaf intro} and~\ref{sect shvs} for more background and precise statements.)

Let $\cS=\cQ^!$ or $\cS=\cD$ denote either of these sheaf theories.
%For a proper map $f:X\to Y$, we have the  pushforward $f_! \simeq f_*:\cS(X)\to \cS(Y)$,
%with right adjoint $f^!:\cS(Y)\to \cS(X)$.
%the sheaf theory that assigns either ind-coherent sheaves or $\D$-modules to a derived stack.

We let $\omega_X =\pi_X^!\cO_{\Spec k}\in \cS(X)$ denote the appropriate dualizing sheaf. Thus for ind-coherent sheaves, $\omega_X\in \cQ^!(X)$ is the algebraic dualizing sheaf, and for $\D$-modules, $\omega_X\in \D(X)$ is the Verdier dualizing sheaf. 
Let $\omega(X)=\pi_{X*}\omega_X$ denote the corresponding  complex of global volume forms:
for ind-coherent sheaves, $\omega(X)\in k\module$ consists of algebraic  volume forms, and for $\D$-modules, $\omega(X)\in k\module$ consists of locally constant distributions (Borel-Moore chains) for $X$ a scheme. 

For $X$ a stack we use the continuous ``renormalized" pushforward functor on $\cD$-modules of~\cite{finiteness}, which roughly replaces equivariant cohomology (derived invariants) by a shift of equivariant homology (derived coinvariants, see~\cite[Example 9.1.6]{finiteness}), so that $\omega(Y/G)$ for a $G$-variety $Y$ is given by a shift of $G$-coinvariants on Borel-Moore chains on $Y$. (Note that even for $X=BG$ this differs from the standard definition of equivariant Borel-Moore homology, which is identified in this case with equivariant cohomology.) 

For a proper (or ind-proper) map $f:X\to Y$, adjunction provides an integration map $\int_f:\omega(X)\to \omega(Y)$.

\begin{example}
Let us continue with the special cases of Example~\ref{intro ex}, and focus in particular on algebraic distributions
$\omega_{\cL X} \in \cQ^!(\cL X)$
 on the loop space.

When $X$ is a smooth scheme, $\cL X \simeq  \BT_X[-1]$ is naturally Calabi-Yau, 
and its global volume forms are identified  with differential forms 
$\omega( \BT_X[-1]) \simeq \cO( \BT_X[-1]) \simeq \Sym(\Omega_X[1])$.
% (with the loop rotation on $\cL X$ providing the de Rham differential).
The canonical  ``volume form" on $\cL X$ is given
by the Todd genus (as explained by Markarian \cite{markarian}): the resulting integration
of functions on $\cL X$ differs from the integration of differential forms on $X$ by the Todd genus.

When $Y =BG$ is a classifying stack,  $\cL Y \simeq G/G$ is  naturally Calabi-Yau,
 and its global volume forms are invariant functions $\omega(G/G) \simeq \cO(G/G) \simeq \cO(G)^G$. If $G$ is reductive with Cartan 
 subgroup $T\subset G$ and Weyl group $W$, the naive invariants $\cO(G)^G \simeq \cO(T)^W$ are equivalent to the derived invariants, but in general there may be higher cohomology.
 \end{example}

\begin{thm}\label{main} Let $\cS=\cQ^!$ or $\cS=\cD$ denote either the theory of ind-coherent sheaves or 
$\D$-modules. Recall our conventions for stacks and morphisms, Setting~\ref{stack conventions}.

$\bullet$ For a stack $X$, there is a canonical identification
$HH_*(\cS(X))\simeq \omega(\cL X)
$
of the Hochschild homology of sheaves on $X$ with distributions (or renormalized Borel-Moore chains) on the loop space.

\medskip

$\bullet$ For a proper (or ind-proper) map of stacks $f:X\to Y$ 
the induced map $HH_*(\cS(X))\to HH_*(\cS(Y))$ is given by integration along the loop map $\cL f:\cL X\to \cL Y$.

\medskip
$\bullet$ {\bf Grothendieck-Riemann-Roch:} In particular, for 
any compact object $M\in \cS(X)$ (coherent sheaf or safe coherent $\cD$-module) with character $[M]\in HH_*(\cS(X))\simeq \omega(\cL X)$, 
there is a canonical identification
$$[f_*\cM]\simeq\int_{\cL f}[M]\in HH_*(\cS(Y))\simeq\omega(\cL Y)$$ 
In other words, the character of a pushforward along a proper map is the integral
of the character along the induced loop map.

\medskip

$\bullet$ {\bf Atiyah-Bott-Lefschetz:} Let $G$ be an affine group, and $X$ a proper stack with $G$-action, or equivalently, a proper map $f:X/G\to BG$.
Then for any compact object $M\in \cS(X/G)$ ($G$-equivariant coherent sheaf or safely equivariant coherent $\cD$-module on $X$), and element $g\in G$,
there is a canonical identification
$$[f_*M]|_g \simeq \int_{\cL f} [M]|_{X^g}
$$ 
In other words, under the identification of invariant functions and volume forms on the group, the value of the character of an induced representation at a group element is given by the integral
of the original character along the corresponding fixed point locus of the group element.

\medskip

$\bullet$ {\bf Extension to traces:} The trace $Tr(\cS(Z))$ of the endofunctor of $\cS(X)$ given by a self-correspondence (e.g. a self-map) $X\leftarrow Z\rightarrow X$ is given by distributions on the fixed points $\omega(Z|_{\Delta})$.

$\bullet$ For a map $f:(X,Z)\to (Y,W)$ of stacks with self-correspondences\footnote{i.e., we lift $f$ to an identification 
$Z\simeq X\times_Y W$ of correspondences from $X$ to $Y$}, the induced map $Tr(\cS(Z))\to Tr(\cS(W))$ is given 
by integration along fixed points $Z|_{\Delta_X}\to W_{\Delta_Y}$.

\end{thm}

%\footnote{among applications: Seidel infinitesimally equivariant GRR}

\begin{example} [{\bf Frobenius-Weyl Character Formula}]

Here is a reminder of a well-known application of the Atiyah-Bott-Lefschetz formula in representation theory.

$\bullet$ If $G$ is a finite group,  and $X=G/K$ is a homogeneous set, and $M=k[G/K]$ the ring of functions, one recovers the Frobenius character formula for 
the induced representation $k[G/K]$.

$\bullet$ If $G$ is a reductive group, $X=G/B$ is the flag variety, $X/G=pt/B\to pt/G$.
The loop map $$\cL(X/G)=B/B\simeq \wt{G}/G\to G/G$$ is the (group) Grothendieck-Springer simultaneous resolution, with fibers giving fixed point loci on the flag variety. For $M=\cL$ an equivariant line bundle on $G/B$, and $g\in G$ runs over a maximal torus, one recovers the Weyl character formula for the induced representation $H^*(G/B, \cL)$.
\end{example}

\begin{remark}
The reader will note no explicit appearance of the Todd genus in the above formulas. 
In other words as for K-theory, pushforward of sheaves naturally agrees with the pushforward in Hochschild homology. The Todd genus arises in comparing these natural pushforwards with the pushforward in cohomology, i.e., integration of forms.
It arises when one unwinds the integration map $
\int_{\cL f}:\omega(\cL X) \to \omega(\cL Y),
$
given by Grothendieck duality, in terms of functions (or differential forms) using the Hochschild-Kostant-Rosenberg theorem.
In particular,  the familiar denominators in the Atiyah-Bott formula are implicit in the integration measure on the fixed point locus.

For instance, as mentioned  above, when $X$ is a smooth scheme, a geometric version of the HKR theorem asserts that the loop space is the total space of the shifted tangent complex  $\cL X \simeq \BT_X[-1]$, and global volume forms are canonically functions
$\omega(\cL X)  \simeq \cO( \BT_X[-1]) \simeq \Sym(\Omega_X[1])$. Under this identification (as explained by Markarian~\cite{markarian}),  
the resulting integration
of functions on $\cL X$ differs from the integration of differential forms on $X$ by the Todd genus.

 \end{remark}

\begin{remark} The paper~\cite{KP} carries out the program described in this paper (i.e. recovering classical identities from non-linear ones)  by calculating explicitly the derived contributions in the case of
the Atiyah-Bott formula.
\end{remark}

\begin{remark}
The main contribution of this paper is hidden in the statement of this theorem:
we establish nonlinear versions of character formulas 
in the setting of derived stacks, and deduce
classical formulas and new higher categorical analogues formally by applying suitable sheaf theories.
Thanks to the great generality of sheaf theory in derived algebraic geometry~\cite{GR}, the resulting applications hold with remarkably few assumptions. 

We are particularly interested in the higher categorical variants where one considers sheaves
of categories, in particular Frobenius-Weyl character formulas for group actions on categories. Since the requisite foundations are not yet fully
developed, we postpone details of this to future works. Applications include an identification of the character of 
 the category of $\D$-modules on the flag variety with the Grothendieck-Springer sheaf, and of
the trace of a Hecke functor on the category of $\D$-modules on the moduli of bundles on a curve with the cohomology of a
Hitchin space. 
\end{remark}

\subsection{Inspirations and motivations}
This work has many inspirations. 
It is heavily reliant on the $\infty$-categorical foundations of higher algebra, derived algebraic geometry and sheaf theory due to Lurie~\cite{HA,SAG} and Gaitsgory-Rozenblyum~\cite{GR}. It is also inspired by Lurie's cobordism hypothesis with singularities~\cite{TFT}, which provides a powerful unifying tool for higher
algebra. Already in the setting of one-dimensional field theory, this result can be viewed 
as a vast generalization of the classical theory Hochschild and cyclic homology and characters therein \cite{Loday}, 
(in particular the natural cyclic symmetry of Hochschild homology
is generalized to a circle action on the dimensions of arbitrary dualizable objects). 
 In particular, the formal properties of traces we use are simple instances of the
cobordism hypothesis with singularities on marked intervals and cylinders.
The work of To\"en and Vezzosi~\cite{TV} on traces and higher Chern characters of sheaves of categories (and in particular the role of the cobordism hypothesis therein) has also profoundly influenced our thinking. 

Another important inspiration is the
categorical theory of strong duality, dimensions and traces introduced by Dold and Puppe in \cite{doldpuppe} (see \cite{may,PS}
for more recent developments) with the 
express purpose of proving
Lefschetz-type formulas. In \cite{doldpuppe}, dualizability of a space is achieved by
linearization (passing to suspension spectra), while our approach is to pass to categories of correspondences (or
spans) instead. We were also inspired by the preprint~\cite{markarian} and the subsequent work
\cite{Cal1,Cal2,Ram,Ram2,Shk}. There have been many recent papers 
\cite{Petit, lunts, polishchuk, cisinskitabuada} building on related ideas to prove Riemann-Roch
and Lefschetz-type theorems in the noncommutative context of differential graded categories and
Fourier-Mukai transforms; our work instead places these results in the context of
the general formalism of traces in $\oo$-categories, and generalizes them to commutative but
nonlinear settings.

The Grothendieck-Riemann-Roch type applications in this paper concern the character map taking coherent sheaves to
classes in Hochschild homology (or in a more refined version, to cyclic homology). This is significantly coarser than the
well established theory of Lefschetz-Riemann-Roch theorems valued in Chow groups (see the seminal \cite{Thomason},
the more recent \cite{joshua} and many references therein).  
Thus for schemes, the quantities compared are Dolbeault (or de Rham) cohomology classes rather than algebraic cycles or $K$- (or rather $G$-)theory classes. 

Our primary motivation is the development of foundations
for ``homotopical harmonic analysis'' of group actions on categories,
aimed at decomposing derived categories of sheaves (rather than classical
function spaces) under the actions of natural operators. This undertaking
follows the groundbreaking path of Beilinson-Drinfeld within the geometric
Langlands program and is consonant with general themes in geometric
representation theory. 
The pursuit of a geometric analogue of the Arthur-Selberg trace formula
by Frenkel  and Ng\^o~\cite{FN} has also been a source of inspiration and applications.

\begin{remark}
A companion paper~\cite{secondary} presents an alternative approach to Atiyah-Bott-Lefschetz formulas
(and in particular a conjecture of Frenkel-Ng\^o) as a special case of the ``secondary trace formula"
identifying trace invariants associated to two commuting endomorphisms of a sufficiently dualizable object.
This is also applied to establish the symmetry of the 2-class functions on a group constructed
as the 2-characters of categorical representations.
\end{remark}

\subsection{Acknowledgements}

We would like to thank Dennis Gaitsgory and Jacob Lurie 
for providing both the foundations and the inspiration for this work, as well as helpful comments
and specifically D.G. for discussions of~\cite{GR}.
We would also like to thank Toly Preygel for many discussions about derived algebraic geometry.

We gratefully acknowledge the support of NSF grants DMS-1103525 (D.BZ.)
and DMS-1319287 (D.N.).

\section{Overview}\label{overview}

\subsection{Traces in category theory}
We highlight structures arising in the general theory of
dualizable objects in symmetric monoidal higher categories (see also \cite{doldpuppe,may,PS}). For
legibility, we suppress all $\oo$-categorical
notations and complications from the introduction. We rely on~\cite{GR} for the theory of symmetric monoidal 
$(\oo,2)$-categories, though only the formal outline of the theory is in fact needed for this paper. See~\cite{TV, HSS} for thorough treatments of the theory of traces in higher category theory.

The basic notion in the theory is that of {\em dimension} of a
dualizable object of a symmetric monoidal category $\cA$. By
definition, for such an object $A$ there exists another $A^\vee$
together with a  coevaluation map $\eta_A$ and evaluation map $\epsilon_A$ satisfying standard
identities.
By definition, the dimension of $A$ is the endomorphism of the  the unit $1_\cA$ given by the  composition
$$\xymatrix{ 1_\cA\ar[rr]^-{\eta_A}\ar@/_2pc/^-{\dim(A)}[rrrr] &&
  A\ot A^\vee \ar[rr]^-{\epsilon_A}&& 1_\cA}$$

\begin{example}
For $V$ a vector space, $V^\vee =\Hom_k(V, k)$ is the vector space of functionals, $\epsilon_V: V\otimes V^\vee \to k$
is the usual evaluation of functionals, $\eta_V:k\to \End(V)\simeq V\ot V^\vee$ is the
identity map (which exists only for $V$
finite-dimensional), and $\dim(V)$ can be regarded as an element of the ground field (by evaluating it on the multiplicative unit).
\end{example}

\begin{remark}[Duality and naiv\"et\'e in $\oo$-categories.]
It is a useful technical  observation that the notion of dualizability in the setting of $\oo$-categories is a ``naive" one: it
is a property of an object that can be checked in the underlying homotopy category. As a result, all of the 
categorical and 2-categorical calculations in this paper are similarly naive and explicit (and analogous to familiar unenriched categorical assertions), 
involving only small amounts of data that can
be checked by hand. 
%We restrict ourselves only to assertions of this naive and accessible nature, specifying all maps that are needed rather than constructing
%higher coherences (for which we view the cobordism hypothesis with singularities as the proper setting). 
\end{remark}

The notion of dimension is a special case of the {\em trace} of an endomorphism $\Phi$
of a dualizable object $A$. By definition,  the trace of $\Phi$ is the endomorphism of the unit $1_\cA$
given by  the composition
$$\xymatrix{1_{\cA}\ar[r]^-{\eta_A}\ar@/_2pc/^-{\Tr(\Phi)}[rrrr]& A \otimes A^\vee\ar[rr]^-{\Phi \otimes \id_{A^\vee}} &&  A\ot A^{\vee}
\ar[r]^-{\epsilon_A}& 1_\cA}$$ 
which recovers the dimension for $\Phi = \id_{A}$. 

A key feature of dimensions and traces is their {\em cyclicity}, which at the coarsest level is expressed by
a canonical equivalence $$\xymatrix{m(\Phi,\Psi):\Tr(\Phi\circ\Psi)\ar[r]^-{\sim}& \Tr(\Psi\circ\Phi),}$$ see Proposition \ref{cyclic trace}. 
 At a much deeper level, an important corollary of the cobordism hypothesis \cite{TFT} is the existence of an 
 $S^1$-action on $\dim(A)$ for any dualizable object $A$ (and an analogous structure for general traces, see Remark \ref{full trace}).

\begin{remark}[Dimensions and traces are local]\label{dim local}
It is  useful for applications to note that the notion of dualizability and the definition of dimension and are  local in the category $\cA$. 
Namely, they only require knowledge of the objects $1_{\cA},A,A^\vee, A\ot A^{\vee}$,
the morphisms $\eta_A,\epsilon_A$, and standard  tensor product and composition identities among them.
Likewise, the notion
of trace only requires the additional endomorphism $\Phi$ along with a handful of additional identities. 
\end{remark}

\subsubsection{Functoriality of traces}
Now suppose the ambient  symmetric monoidal category $\cA$ underlies a 2-category, so there is the
possibility of noninvertible 2-morphisms. This allows for the
notion of left and right adjoints to morphisms. Let us say a morphism
$A\to B$ is {\em continuous}, or {\em right dualizable}, if it has a right
adjoint. (The terminology derives from the setting of presentable categories,
where the adjoint functor theorem guarantees the existence of right adjoints
for colimit preserving functors.)

Here are natural functoriality properties of
dimensions and traces.

\begin{prop} Let $A,B$ denote dualizable objects of $\cA$ and $f_*:A\to B$ a continuous morphism with right adjoint $f^!$.
\begin{enumerate}

\item There is  a canonical
map on dimensions
$$\xymatrix{\dim(A)\ar@/_2pc/^-{\dim(f_*)}[rrrrr] \ar[r]^-{=}&\Tr(\Id_A) \ar[r] & \Tr(f^! f_*) \ar[r]^-\sim & \Tr(f_* f^!) \ar[r] & \Tr(\Id_B)\ar[r]^-{=}&\dim(B)}$$
compatible with compositions of continuous morphisms.

\item Given endomorphisms $\Phi\in\End(A)$, $\Psi\in \End(B)$, and a {\em commuting structure}
$$\xymatrix{\alpha:f_*\circ\Phi\ar[r]^-{\sim}& \Psi\circ f_*}$$
there is a canonical map on traces
$$
\xymatrix{
\Tr(f_*,\alpha): \Tr(\Phi)\ar[r] & \Tr(\Psi)
}
$$
compatible with compositions of continuous morphisms with commuting structures.
\end{enumerate}
\end{prop}

We refer to the compatibility with compositions stated in the proposition as {\em abstract Grothendieck-Riemann-Roch}.
To see its import more concretely, let us restrict the generality and focus on an {\em object} of $A$ in the sense of a
morphism $V:1_\cA\to A$.

\begin{corollary}
 Let $A,B$ denote dualizable objects of $\cA$ and $f_*:A\to B$ a continuous morphism.
For $V:1_\cA\to A$ an object of $A$, we obtain a map on dimensions
$$
\xymatrix{\dim(V):1_\cA \simeq \dim(1_\cA)\ar[r] & \dim(A)
}
$$ 
called the {\em character} of $V$ and alternatively denoted by $[V]$.
It satisfies abstract
Grothendieck-Riemann-Roch in the sense 
% $
% (\dim f_*)[V]=[f_*V]
% $ or in other words, 
that the following diagram commutes
$$\xymatrix{ 1_\cA\ar[rr]^-{[V] }\ar@/_2pc/^-{[f_*V]}[rrrr] &&
  \dim(A)  \ar[rr]^-{\dim(f_*)}&& \dim(B)}$$

\end{corollary}

\begin{remark}[Functoriality of dimensions and traces is local]\label{functoriality local}
As in Remark \ref{dim local}, it is useful to note that the functoriality of dimension is  local, depending only
on a handful of objects, morphisms and identities, along with
the additional adjunction data $(f_*,f^!)$. A similar observation applies to the functoriality of traces.
\end{remark}

\begin{remark}
It follows from the cobordism hypothesis with singularities \cite{TFT} (see~\cite{TV}) that the morphism $\dim(f_*)$  is
$S^1$-equivariant, and hence the character $[V]$ is $S^1$-invariant, though we will not elaborate on this structure here. We refer to~\cite{HSS} for a thorough study of the functoriality and cyclicity of traces.
\end{remark}

\begin{example} Let $dgCat_k$ denote the symmetric monoidal $\oo$-category of presentable $k$-linear differential graded categories (or alternatively, stable presentable $k$-linear $\oo$-categories), see e.g.~\cite{GR}.
In this setting, any compactly generated category $A$ is dualizable, and its dimension
 is the Hochschild chain complex $\dim(A)=HH_*(A)$. The $S^1$-action on $\dim(A)$ corresponds to Connes' cyclic structure
 on $HH_*(A)$,
so that in particular, the localized $S^1$-invariants of $\dim(A)$ form the periodic cyclic homology of $A$.  

More generally, the trace of an endofunctor $\Phi:A\to A$ is  the Hochschild homology
$\Tr(\Phi)=HH_*(A,\Phi)$.  For example, if $A=R\module$ for a
dg algebra $R$, then $\Phi$ is represented by an $R$-bimodule $M$, and we recover the 
Hochschild homology $HH_*(R,M)$.

Any compact object  $M\in A$ defines a continuous functor 
$$\xymatrix{1_{dgCat_k}=dgVect_k\ar[r]^-{M} & A}$$
whose  character is a vector 
$$\dim(M)\in HH_*(A)
$$ in Hochschild homology (with refinement in cyclic homology). 
The abstract Grothendieck-Riemann-Roch theorem
expresses the natural functoriality of characters   in Hochschild homology (or their refinement in cyclic homology). 
In fact, the construction of characters factors through the canonical Dennis
trace map 
$$
\xymatrix{
A_{cpt}\ar[r] & K(A)\ar[r] & HH_*(A)
}$$ from the space  $A_{cpt}$ of compact objects of $A$. 
\end{example}

%%%

\subsection{Traces in geometry}
To apply the preceding formalism to geometry, it is useful to organize spaces and maps within a suitable categorical framework.
We then arrive at loop spaces and fixed point loci as
{nonlinear} expressions of dimensions and traces.
This simple observation provides the core of the paper.
 Throughout the discussion, we continue to suppress all $\oo$-categorical
notations and complications. Our reference for the correspondence 2-category of stacks is Section V of~\cite{GR} (see also~\cite{haugsengspan}).

To begin, consider the general setup of the symmetric monoidal category $Corr(\cC)$ of correspondences, or spans, in a category $\cC$ such as stacks (or formally a symmetric monoidal $\infty$-category with finite limits; see~\cite{barwick,haugsengspan}). Here the objects $X\in Corr(\cC)$ are the objects of $\cC$, the morphisms $Corr_\cC(X,Y)$ are arbitrary spans in $\cC$,
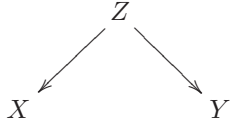
$$\xymatrix{
&Z\ar[ld]\ar[rd]&\\
X&&Y}$$ (more generally one can require the left and right legs to live in specified subcategories of $\cC$ as in~\cite{GR}).
The composition of morphisms $Z\in Corr_\cC(X,Y)$ and $W\in Corr_\cC(Y,U)$ is given by fiber product
$$
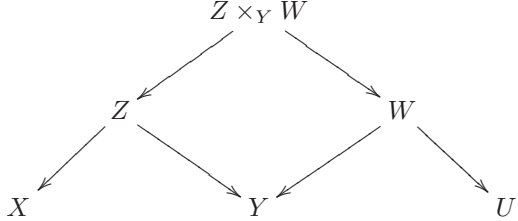
\xymatrix{&&Z\times_Y W\ar[ld]\ar[rd]&&\\
&Z\ar[ld]\ar[rd]&&W\ar[ld]\ar[rd]& \\
X&&Y&&U}$$
and the symmetric monoidal structure is given in terms of that on $\cC$ (Cartesian product in the case of stacks).
For the purpose of calculating dimensions and traces, we need not require any
further properties of the spaces of $Corr(\cC)$, since we need only 
the modest local data discussed in Remarks~\ref{dim local} and~\ref{functoriality local}.
(See \cite{TFT} and \cite{FHLT}, where the higher categories $Fam_n$ of iterated correspondences of manifolds 
are constructed
and applied.)

With applications in mind, we will specialize to the correspondence category $Corr_k=Corr(St_k)$ of 
derived stacks over $k$.  It would also be interesting to work
with smooth manifolds instead, for example 
through the theory of $C^\infty$-stacks~\cite{joyce} (see Remark \ref{Cinfty}).

It is natural to enhance $Corr(\cC)$ to a 2-category $\ul{Corr}(\cC)$ by
 allowing non-invertible maps, or more generally correspondences, between correspondences (see Section V of~\cite{GR} or~\cite{haugsengspan} for full details), so that maps from $X$ to $Y$ form the category of objects over $X\times Y$. 
 Our constructions naturally fit into
 the 2-category
 $\ul{Corr}^{prop}_k$ with non-invertible 2-morphisms restricted to be {\em proper} (or more generally ind-proper) maps 
 of correspondences.

\begin{remark}[Correspondences are bimodules]\label{correspondences are bimodules}

It is useful to view the correspondence category $Corr$ within the framework of coalgebras in symmetric monoidal categories.
The diagonal map $X\to X\times X$ makes any space or stack into a cocommutative coalgebra object with respect
to the Cartesian product monoidal structure (or commutative coalgebra in the opposite category). Moreover, a map $Z\to X$ is equivalent
to an $X$-comodule structure on $Z$.
Thus correspondences from $X$ to $Y$ may be interpreted as $X-Y$-bicomodules, with  composition of correspondences
given by tensor product of bicomodules. 

Furthermore, it is natural to enhance $Corr$ to a  2-category by allowing non-invertible 
maps between correspondences. 
This can be viewed 
as a special case of
the Morita category of coalgebras in a symmetric monoidal category.
The 2-category $\ul{Corr}$ of spaces, correspondences, and maps of correspondences  is the Morita
category on spaces regarded as coalgebra objects.
(In particular, the cocommutativity of the coalgebra objects implies they are canonically self-dual, and the transpose
of a correspondence is the same correspondence read backwards.) 
If we further keep track of the $E_n$-coalgebra structure of spaces and consider the corresponding Morita $(n+1)$-category, we recover
the $(n+1)$-category of iterated correspondences of correspondences. (See~\cite{haugsengspan,haugsengEn} for a thorough treatment of categories of spans and Morita categories of $E_n$ algebras; 
see also for example the category $Fam_n$ of \cite{TFT} and \cite{FHLT} in the topological setting.)
\end{remark}

\subsubsection{Geometric dimensions and loop spaces}

A crucial feature of the category $Corr_k$  is that any object $X\in Corr_k$ is dualizable (in fact, canonically self-dual), thanks
to the diagonal correspondence.\footnote{Likewise, if we wish to make a space $n$-dualizable for any $n$ we may 
simply consider it as an object of a higher correspondence category as in Remark~\ref{correspondences are bimodules},
since $E_n$-(co)algebras are $n+1$-dualizable objects of the corresponding Morita category. In other words, a space $X$ defines a 
topological field theory of any dimension 
valued in the appropriate correspondence category.} Note for this it is crucial that we allow all maps, including the map $\pi_X:X\to pt$ for any $X$, as possible legs in a span.

We have the following calculations of dimensions and their functoriality. Note that the point $pt=\Spec k$ is the unit of $Corr_k$.
We keep track of properness of maps of correspondences for the later application of sheaf theory.

\begin{prop} Let $Corr_k$ be the category of derived stacks and correspondences, and $\ul{Corr}^{prop}_k$ the 2-category of derived stacks, correspondences, and 
proper maps of correspondences.

(1) Any derived stack $X$ is dualizable as an object of $Corr_k$, and its dimension $\dim(X)$ is identified 
with the loop space 
$$
\cL X= X^{S^1}\simeq X\times_{X\times X} X
$$ 
 regarded as a self-correspondence of $pt=\Spec k$.

(2) A map $f:X\to Y$ regarded as a  correspondence from $X$ to $Y$
is continuous in $\ul{Corr}^{prop}_k$ if and only if $f$ is proper. 
Given a proper map $f:X\to Y$, its induced   map 
$$
\xymatrix{
\dim(f):\dim(X)\ar[r] & \dim(Y)
}
$$ is identified with the
loop map 
$$
\xymatrix{
\cL f:\cL X \ar[r] & \cL Y
}
$$

\end{prop}

\begin{remark}
All of the objects and maps of the proposition have natural $S^1$-actions, on the one hand coming from loop rotation, on the other hand coming from the cyclic symmetry of dimensions. One can check that the identifications of the proposition are $S^1$-equivariant
(see Remark~\ref{rem: cyclic}).
\end{remark}

\begin{remark}
Recall \cite{Toen,conns} that for a 
 derived scheme $X$, 
the loop space $\cL X\simeq T_X[-1]$ is the total space of the shifted tangent complex.
The action map of the $S^1$-rotation action is encoded by the de Rham differential.
For an underived stack $X$, the loop space is a derived enhancement of the inertia stack
$
IX=\{x\in X, \gamma\in Aut(x)\}.
$
The action map of the $S^1$-rotation action is manifested by
 the ``universal automorphism" of any sheaf on $\cL X$.
\end{remark}

\begin{example}
Let $G$ denote an algebraic group and $BG=pt/G$ its classifying space.
There is a canonical identification $\cL BG\simeq G/G$ of the loop space and adjoint quotient.

Suppose we are given a  $G$-derived stack $X$, or equivalently a  morphism $\pi:X/G\to BG$, from which one recovers $X\simeq X/G \times_{BG} pt$. 
(Note that if we want $\pi$ proper we should take $X$ itself proper.)

Let us explain how the loop map $\cL \pi:\cL(X/G) \to \cL(BG)$ captures the fixed points of $G$ acting on $X$. For any self-map $g: X\to X$, let us write  $X^g$ for the derived fixed point locus given by the derived intersection
$$
X^g= \Gamma_g \times_{X\times X} X.
$$
of the graph $\Gamma_g\subset  X\times X$ with the diagonal. Then $\cL \pi$ map fits into a commutative square
$$
\xymatrix{
\ar[d]_-\sim \cL(X/G)  \ar[r]^-{\cL \pi} & \cL(BG) \ar[d]_-\sim \\
 \{g\in G, x\in X^g\}/G \ar[r]^-p & G/G
}
$$
where $p$ projects to the group element.

In particular, fix a group element $g\in G$, with conjugacy class $\OO_g\subset G$, and centralizer $\cZ_G(g)\subset G$,
so that  $\OO_g/G\simeq B\cZ_G(g) \in G/G$. Then 
the corresponding fiber of $\cL \pi$ is 
the equivariant fixed point locus  $X^g_G=X^g/\cZ_G(g)$, or in other words we have a fiber diagram
$$
\xymatrix{
\ar[d]  X^g_G \ar[r] & \OO_g/G \ar[d] \\
 \cL(X/G)  \ar[r]^-{\cL \pi} & \cL(BG)  
 }
$$

Let us specialize to the case of a subgroup $K\subset G$, and the quotient $X=G/K$, so that we have a map of classifying stacks $\pi: BK \simeq G\backslash (G/K)\to BG$.
Here the loop map $\cL \pi$ realizes the familiar geometry of the Frobenius character formula
$$
\xymatrix{
\ar[d]_-\sim \cL(BK) \ar[r]^-{\cL \pi} &  \cL(BG)\ar[d]_-\sim\\
  K/K \simeq  \{g\in G, x\in (G/K)^g\}/G
 \ar[r]^-p &  G/G
 }
 $$
The equivariant fixed point loci express the equivariant inclusion of conjugacy classes.

Specializing further, for $G$ a reductive group, $B\subset G$ a Borel subgroup, and $X=G/B$ the flag variety, we recover the group-theoretic Grothendieck-Springer resolution
$$
\xymatrix{
\ar[d]_-\sim \cL(BB) \ar[r]^-{\cL \pi} &  \cL(BG)\ar[d]_-\sim\\
  B/B \simeq \{g\in G, x\in (G/B)^g\}/G
 \ar[r]^-p &  G/G
 }
 $$
\end{example}

\subsubsection{Geometric traces of correspondences}
More generally, we have the following calculations of  traces and their functoriality. 

\begin{prop}
Let $Corr_k$ be the category of derived stacks and correspondences, and $\ul{Corr}^{prop}_k$ the 2-category of derived stacks, correspondences, and proper maps of correspondences.

(1) The trace of a self-correspondence $Z\in \ul{Corr}^{prop}_k(X,X)$ is its fiber product with the diagonal
$$
\Tr(Z) \simeq Z|_{\Delta}=Z\times_{X\times X} X\simeq Z\times_X\cL X
$$ 

In particular, for the graph $\Gamma_f\to X\times X$ of a self-map $f:X\to X$, its trace is the fixed point locus of the map
$$
\Tr(\Gamma_g)  \simeq X^f=\Gamma_f\times_{X\times X} X
$$

(2) Given a proper map $f:X\to Y$ regarded as a correspondence from $X$ to $Y$, and self-correspondences $Z\in\ul{Corr}^{prop}_k(X, X)$ and $W \in\ul{Corr}^{prop}_k(Y, Y)$, together with an identification 
$$
\xymatrix{
\alpha: Z  \ar[r]^-\sim  & X\times_Y W  
}$$ 
of correspondences from $X$ to $Y$,
the induced abstract trace map 
$$
\xymatrix{
\Tr(f,\alpha):\Tr(Z) \ar[r] & \Tr(W)
}
$$ 
is equivalent to the induced geometric map 
$$
\xymatrix{
\tau(f, \alpha): Z|_{\Delta_X}\ar[r] &  W|_{\Delta_Y}
}$$
\end{prop}

\subsection{Trace formulas via sheaf theories}\label{sheaf intro}

%$\bullet$ The secondary trace of two commuting self-correspondences is
%given by the loop space of the trace of the composition. (?)

 Given any sufficiently functorial
method of measuring derived stacks, 
the preceding calculations of geometric dimensions, traces and their functoriality immediately lead to trace and character formulas. 
To formalize the functoriality needed, we will use the language of sheaf theories. 
Broadly speaking, a sheaf theory is a representation (symmetric monoidal functor out) of a correspondence category in the way
 a topological field theory is a representation of a cobordism category.\footnote{Indeed a typical mechanism to construct ``Lagrangian" field theories is as the composition of a sheaf theory with a ``classical field theory" as in~\cite{FHLT,haugsengspan}, a symmetric monoidal functor from a cobordism category to a correspondence category.}
 It provides an approach to encoding the standard operations
on coherent sheaves and $\D$-modules, developed by Gaitsgory and Rozenblyum in the book~\cite{GR} (following a suggestion of Lurie and previous versions in~\cite{koszul,indcoh,finiteness,crystals}).

We will take the target of our sheaf theories to be the linear setting of the 
symmetric monoidal $(\oo,2)$-category $\ul{dgCat}_k$ of presentable $k$-linear differential graded categories with continuous functors and natural transformations. 
(Recall from Setting~\ref{stack conventions} that for applications all stacks are assumed to be QCA or ind-inf-schemes. Also all proper maps can be replaced by ind-proper ones at no cost.)

\begin{defn}\label{sheaf def intro} A {\em sheaf theory} is a symmetric monoidal functor of $(\infty,2)$-categories
$$
\xymatrix{
\ul{\cS}:\ul{Corr}_k^{prop}\ar[r] &  \ul{dgCat}_k
}
$$ 
from correspondences of stacks (with 2-morphisms given by proper maps of correspondences) to dg categories. We denote by 
$$
\xymatrix{
\cS:Corr_k\ar[r] &  dgCat_k
}
$$ 
the underlying 1-categorical sheaf theory, i.e. the symmetric monoidal functor on $(\infty,1)$-categories obtained by forgetting noninvertible morphisms.
\end{defn}

Let us first spell out some of the structure encoded in a 1-categorical sheaf theory $\cS$.

The graph of a map of stacks $f:X\to Y$ provides a correspondence from $X$ to $Y$ and a correspondence from $Y$ to $X$. We denote the respective induced maps by 
$f_{*}:\cS(X)\to \cS(Y)$ and $f^!:\cS(Y)\to \cS(X)$.
 For $\pi:X\to pt = \Spec k$, we denote by  $\omega_{X}=\pi^! k\in \cS(X)$ the 
$\cS$-analogue of the dualizing sheaf, and by
$\omega(X)=\pi_{*}\omega_{X} \in \cS(pt) = dgVect_k$ the $\cS$-analogue of  ``global volume forms". 
We
adopt traditional notations whenever possible, for example writing
$\Gamma(X, \cF)=\pi_{*}(\cF)$, for $\cF\in \cS(X)$

The functoriality of $\cS$ concisely encodes base change for $f_{*}$ and $f^!$. Its symmetric monoidal structure provides equivalences
$$\cS(X\times Y)\simeq \cS(X)\ot \cS(Y),$$ as well as a symmetric monoidal structure on $\cS(X)$ for any $X$ (using pullback along diagonal maps). The 2-categorical extension $\ul{\cS}$ further encodes an identification of $f^!$ with the right adjoint of $f_*$ for $f$ proper.

Since a sheaf theory $\cS$ is symmetric monoidal, it is automatically compatible with dimensions and traces: for any $X\in Corr_k$,
and any endomorphism $Z\in Corr_k(X,X)$, we have
$$
\xymatrix{
\dim (\cS(X)) \simeq \cS(\dim (X))   & \Tr(\cS(Z)) \simeq \cS(\Tr(Z)) 
}$$
Let us combine this with the calculation of the right hand sides and highlight specific examples of interest.

\begin{prop}\label{intro HH description} 
Fix a sheaf theory $\cS:Corr_k\to dgCat_k$.

(1) The $\cS$-dimension $\dim(\cS(X)) =HH_*(\cS(X))$ of any $X\in Corr_k$  is $S^1$-equivariantly equivalent with $\cS$-global volume forms on
the loop space 
$$
\dim(\cS(X))\simeq \omega(\cL X).
$$

 In particular,  for $G$ an affine algebraic group, characters
of $\cS$-valued $G$-representations are adjoint-equivariant $\cS$-global volume forms
 $$\dim(\cS(BG))\simeq \omega(G/G)
 $$

(2)  The $\cS$-trace of any endomorphism $Z\in Corr_k(X,X)$ is equivalent to $\cS$-global volume forms
on the restriction to the diagonal
$$
\Tr(\cS(Z)) \simeq \omega(Z|_{\Delta})
$$

 In particular, the $\cS$-trace of a self-map $f:X\to X$  is equivalent to $\cS$-volume forms on the $f$-fixed point locus
$$
\Tr(f_{*})\simeq \omega(X^f)
$$
\end{prop}

\begin{remark}[Local sheaf theory]\label{local sheaf theory} To apply this proposition, far less structure than  a 
full sheaf theory is required. We only need the data of the functor $\cS$ on the handful 
of objects and morphisms involved in the construction of dimensions and traces as in Remark \ref{dim local}. In particular,
we only need base change isomorphisms for pullback and pushforward along specific diagrams, rather than the general  base change provided by a functor out of $Corr_k$.
This is often easy to verify in practice, in particular for the examples $\cQ$, $\cQ^!$ and $\D$ (see for example~\cite{BFN} for 
the quasicoherent setting).
\end{remark}

\subsubsection{Examples of sheaf theories}
As we explain in Section~\ref{sect shvs}, the work~\cite{GR} (combined with essential results from~\cite{finiteness}) construct two sheaf theories $\cQ^!$ and $\cD$:

\medskip
$\bullet$ Theory $\cQ^!$: the theory of ind-coherent sheaves $\cQ^!(X)$. 
This is the ``large" version $\cQ^!(X)=\Ind \mathit{Coh}(X)$
of the category of coherent sheaves, which by definition are the compact objects in $\cQ^!(X)$. 
(For smooth $X$, ind-coherent and quasicoherent sheaves are equivalent.) 
Maps are given by the standard pushforward $f_*$ and  exceptional pullback $f^!$. 
The $\cQ^!$-dualizing sheaf is the usual  dualizing complex $\omega_X$, and (for $X$ proper) the $\cQ^!$-global volume forms are its sections
$R\Gamma(X,\omega_X)=R\Gamma(X,\cO_X)^*.$
The $K$-theory of $\cQ^!(X)$ is algebraic $G$-theory $G(X)$, the homological version of algebraic $K$-theory for potentially singular spaces suited to Grothendieck-Riemann-Roch theorems.

\medskip
$\bullet$ Theory $\D$: the theory of $\D$-modules $\D(X)$
with the standard functors $f_*$ and $f^!$. The compact objects are necessarily coherent $\D$-modules (this suffices for $X$ a scheme; see \cite{finiteness} for a characterization in the case of a stack). The $\cD$-dualizing sheaf is the Verdier dualizing complex  $\omega_X$, and the $\cD$-global volume forms (for $X$ smooth) are the Borel-Moore homology
$R\Gamma(X_{dR},\omega_X)=H_{dR}(X)^*$.

\begin{remark} More precisely,~\cite{GR} construct the sheaf theories $\cQ^!$ and $\cD$ as lax symmetric monoidal functors on a much broader class of stacks, with pullbacks allowed for arbitrary maps but pushforward only for schematic morphisms. The strictness follows from results of~\cite{finiteness}, as we explain in Section~\ref{QCA section}, as does the definition of pushforwards for arbitrary maps of QCA stacks (without the functorial apparatus of~\cite{GR} but sufficient for all the ``local" constructions we need, in the sense of Remarks~\ref{dim local},~\ref{functoriality local} and~\ref{local sheaf theory}).
\end{remark}
\medskip

\begin{remark}[Quasicoherent sheaves]  The theory of quasicoherent sheaves $X\mapsto\cQ(X)$ behaves similarly with respect to 1-categorical properties. It also defines a symmetric monoidal functor out of the $(\infty,1)-$category of correspondences of stacks, using standard pullback $f^*$ and pushforward $f_*$ functors. Assuming $X$ is perfect (in the sense of~\cite{BFN}), the compact
objects of $\cQ(X)$ form the subcategory of perfect complexes $\mathit{Perf}(X)$, and we have $\cQ(X) = \Ind\mathit{Perf}(X)$. The analog of the dualizing sheaf is the structure sheaf $\cO_X$, and the $\cQ$-``global volume forms" are the global functions $R\Gamma(X,\cO_X)$. The $K$-theory of $\cQ(X)$ is the usual algebraic $K$-theory $K(X)$. However while we have $(f^*,f_*)$ adjunction and $f^*$ preserves perfection for arbitrary morphisms, proper pushforward does not preserve perfection and we do not have proper adjunction of the form $(f_*,f^*)$ (unless we add smoothness and twisting by relative dualizing sheaves). In other words, $\cQ$ and $K$-theory are better adapted to pullback, while $\cQ^!$ and $G$-theory are better adapted to integration and character formulas.
\end{remark}

\begin{remark}[Sheaf theories in differential topology and elliptic operators] \label{Cinfty}
It is tempting to think of sheaf theories in algebraic geometry as analogues of elliptic operators or complexes in differential topology.
In particular, the theory $\cQ^!(X)$ for a smooth variety $X$ 
is a natural setting for the study of the Dolbeault
$\overline{\del}$-operator coupled to vector bundles, while the theory $\D(X)$ is similarly a natural setting
for the study of the de Rham operator $d$ coupled to vector bundles. The pushforward operation is the analogue of the index. 
In this direction, it would be  interesting to develop sheaf theories on derived manifolds, for example
$C^\infty$-schemes and stacks. Quasicoherent sheaves in the sense of Joyce \cite{joyce} are a natural
candidate. Another interesting setting is categories of elliptic complexes on manifolds. The general results below would then 
provide an approach to generalizations of the classical Atiyah-Singer and Atiyah-Bott theorems.
\end{remark}

\medskip

Let us spell out the main ingredients of Proposition~\ref{intro HH description} for our examples. Recall that for $X$ a smooth scheme, $\cL X\simeq \Spec_X \Sym(\Omega_X[1])$, and for $BG$ a classifying stack, $\cL(BG) \simeq G/G$.

\medskip
$\bullet$ Theory $\cQ$:  For $X$ a smooth scheme,  we have the HKR identification of functions on the loop space (or the Hochschild chain complex) with differential forms, $\dim(\cQ(X)) \simeq \Gamma(X,\Sym(\Omega_X[1]))$, or more generally, $\cQ$-global volume forms on $X^f$ are the coherent cohomology $\cO(X^f)$. For $BG$ a classifying stack,  $\cQ$-global volume forms on $\cL (BG)$ are the coherent cohomology $\cO(G/G)$, which for $G$ reductive are the underived invariants $\cO(T)^W$.

\medskip
$\bullet$ Theory $\cQ^!$: For $X$ smooth, we have $\cQ(X)\simeq \cQ^!(X)$, and so we recover the above descriptions.
For $X$ proper, $\cQ$-global volume forms on $\cL X$ are the dual of the Hochschild chain complex (see \cite{toly}).  
%More generall,  we find a version of Hochschild chains and functions on derived fixed points, twisted by the dualizing sheaf of $X$.

\medskip
$\bullet$ Theory $\D$: 
For $X$ a smooth scheme,  $\D$-global volume forms on $\cL X$ are the de Rham cochains $\dim(\D(X)) \simeq C^*_{dR}(X)$, or more generally, $\D$-global volume forms on $X^f$ are the de Rham cochains $C^*_{dR}(X^f)$, or equivalently those of
  the underlying underived scheme of $X^f$. For $BG$ a classifying stack,  $\D$-global volume forms on $\cL (BG)$ are the Borel-Moore homology of $G/G$.

\subsubsection{Integration formulas for traces}\label{integration section}
Now let us turn to the functoriality of dimensions and traces, which is reflected in integration of volume forms along proper maps. 

For a sheaf theory $$\uS:\uCorr^{prop}_k\to \udg_k,$$ the counit of the $(f_*,f^!)$ adjunction for a proper map $f:X\to Y$
gives rise to a canonical integration map
$$
\xymatrix{
\int_f: \omega(X)\ar[r] & \omega(Y)
}
$$
%
%Finally, we have the functoriality of traces in parallel with the previous theorem on the functoriality of  dimensions.
%Let us recall the relevant setup. Consider a proper morphism $f:X\to Y$ and endomorphisms 
% $F_Z:X\to X$ and $F_W:Y\to Y$ in $Corr_k$ given by respective self-correspondences 
% $X\leftarrow Z \to X$ and $Y \leftarrow W \to Y$. 
%
%  By an $f$-morphism from the pair $(X,F_Z)$ to the pair $(Y,F_W)$, we mean an
%  identification 
%  $$
%  \xymatrix{
%  s:Z  \ar[r]^-\sim &  X\times_Y W
%  }$$ 
%  of correspondences from $X$ to $Y$. This in turn
%  induces an identification of what might be called relative traces
%  $$
%  \xymatrix{
%  Z\times_{Y\times Y} Y\ar[r]^-\sim &  X\times_{Y\times Y} W
%  }
%  $$ 
%  generalizing the relative loop space $\cL_X Y$ from the case of the identity correspondences $Z=X$, $W=Y$.
% We thus obtain a map of traces
%$$
%\xymatrix{
%\tau(f, s):Z|_{\Delta_X}  =   Z\times_{X\times X} X \ar[r] &   Z \times_{Y\times Y} Y \ar[r]^-\sim &  X \times_{Y\times Y} W \ar[r] &  
%  Y\times_{Y\times Y} W = W|_{\Delta_Y}
%}
%$$
%

\begin{thm}  Fix a sheaf theory $\uS:\uCorr_k\to \udg_k$.

(1) For any proper map $f:X\to Y$, the induced map on dimensions 
$$
\xymatrix{
\dim(f_*):\dim(\cS(X))\ar[r] & \dim(\cS(Y))
}
$$
is identified ($S^1$-equivariantly) with integration along the loop map
$$
\xymatrix{
\dim(f_*)\simeq \int_{\cL f}:\omega(\cL X) \ar[r] & \omega(\cL Y)
}
$$

(2) Given a proper map $f:X\to Y$ regarded as a correspondence from $X$ to $Y$, and self-correspondences $Z\in\ul{Corr}_k(X, X)$ and $W \in\ul{Corr}_k(Y, Y)$, together with an identification 
$$
\xymatrix{
\alpha: Z  \ar[r]^-\sim  & X\times_Y W  
}$$ 
of correspondences from $X$ to $Y$,
the induced trace map is identified with integration
along the natural map
$$
\xymatrix{
\Tr(f_*,\alpha) \simeq \int_{\tau(f, s)} : \omega( Z|_{\Delta_X}) \ar[r] & \omega(W|_{\Delta_Y} )
}$$
\end{thm}

\begin{remark}
Similarly, in the case of the theory $\cQ$ of quasicoherent sheaves, the standard adjunction $(f^*,f_*)$ leads to the evident contravariant functoriality of dimensions under arbitrary maps, given by pullback of functions on loop spaces.
\end{remark}

%
%\begin{thm} Let $\cS$ denote either $\cQ^!$ or $\D$.
%\begin{enumerate}
%\item[$\bullet$] For any proper $f:X\to Y$, the induced map on dimensions (Hochschild homologies) $\dim(f_*):\dim(\cS(X))\to \dim(\cS(Y))$ 
%is identified with the ($S^1$-equivariant) morphism given by integration along the loop map:
%$$\dim(f_*)=\int_f:\omega(\cL X) \to \omega(\cL Y) .$$
%
%\item[$\bullet$] {\bf Grothendieck-Riemann-Roch:} For any compact object $\cM\in \cS(X)$ 
%we have an identification
%$$[f_*\cM]=\int_{\cL f}[\cM],$$ in other words, the character of a push forward is given by integration
%of the character along the loop map.
%\medskip
%
%\item[$\bullet$] {\bf Atiyah-Bott-Lefschetz:} Fix an algebraic group $G$ and a complete $G$-space $X$, and let $\pi:X/G\to pt/G$ denote the canonical
%projection. For any compact
%object $M\in \cS(X/G)$ (a
%$G$-equivariant sheaf on $X$), the character $[\pi_*M]$
%of the induced representation $\pi_*M$  is given
%by integration of the Euler class $[M]$ along the equivariant fixed points on $X$: for any conjugacy class $i:\OO_g\to G/G$ we have
%
%$$i^![\pi_*M]= \int_{X_G^g} [M|_{X_G^g}]$$
%
%
%\item[$\bullet$] Given correspondences $Z:X\to X$ and $W:Y\to Y$ together with an identification
%$\alpha:Z\times_X Y\simeq X\times_Y W: X\to Y$, the trace map  is identified with integration along
%the canonical map $Z|_{\Delta_X}\to W|_{\Delta_Y}$:
%$$\Tr(f_*,\alpha)=\int_{Z|_{\Delta_X}\to W|_{\Delta_Y}}:\omega(Z|_{\Delta_X})\to \omega(W|_{\Delta_Y})$$
%
%\end{enumerate}
%\end{thm}

\begin{remark}[Categorified version]
For applications to categorical representation theory, in particular the geometric Langlands program,
it is interesting to have character formulas for group actions on categories. Such formulas would follow from a good formalism of ``stack theories", the higher unstable analogs of sheaf theories, such as the assignment $X\to ShvCat_k(X)$. Such stack theories could be formulated as symmetric monoidal functors $\uS:\uCorr_k\to\cA$ out of a correspondence $(\infty,2)$ (or more naturally $(\infty,3)$) category with values in a category $\cA$ such as that of module categories for $dgCat_k$.
Namely, we are interested in categorified analogues of $\cD$ and $\cQ$, taking values in the $\oo$-category $\Pr^L$ of presentable $\oo$-categories, in which we assign to a scheme or stack $X$ the $\infty$-category
of quasicoherent sheaves of module categories over $\D$ or $\cQ$.  Since such theories have not been fully constructed yet, 
we will only briefly sketch the idea.

For any stack $X$ and sheaf theory $\cS$, the category of sheaves $\cS(X)$ is naturally symmetric monoidal, and so we may
consider its $\oo$-category of (presentable, stable) module categories $\cS(X)\module$. To obtain a more meaningful geometric  theory we should sheafify this construction. For example, strong  or Harish-Chandra $G$-categories (in other words, module categories over $\D(G)$ with convolution) are identified with 
sheaves of categories over the de Rham stack of $BG$. However, in the quasicoherent case, the ``1-affineness" theorem of
Gaitsgory~\cite{1affine} identifies $\cQ(X)$-modules with sheaves of categories on $X$ for a large class of stacks
(specifically, for $X$ an eventually coconnective quasi-compact algebraic stack of finite type with an affine diagonal over a field of characteristic 0). In particular,
 $\cQ(BG)$-modules are identified with algebraic $G$-categories. 

In the quasicoherent case, the general formalism of this paper should  provide an $S^1$-equivariant equivalence $\dim(\cQ(X)\module)=\cQ(\cL X)$, identifying
the class $[\cQ(X)]$ of the structure stack with the structure sheaf $\cO(\cL X)$. In particular, the characters
of quasicoherent $G$-categories are given by $\cQ(G/G)$. The induced map on dimensions  $\dim(f_*):\dim(\cQ(X)\module)\to \dim(\cQ(Y)\module)$ 
is identified $S^1$-equivariantly with the morphism given by pushforward along the loop map
$$
\xymatrix{
\dim(f_*)=\cL f_*:\cQ(\cL X) \ar[r] &  \cQ(\cL Y)
}
$$
In particular, for an algebraic group $G$ and $G$-space $X$ with $\pi:X/G\to BG$, the character of the $G$-category
$\cQ(X/G)$ is given by the pushforward $\cL\pi_* \cO(\cL X/G)\in \cQ(G/G)$. Analogous results are expected
for strong or Harish-Chandra $G$-categories (module categories for $\D(G)$ with convolution) using the sheafification of the theory
of $\D(X)$-module categories. We hope to return to these applications in future works.
\end{remark}

%%%
%%%
%%%

\section{Traces in category theory}\label{sect cat}

\subsection{Preliminaries}
Our working setting is the higher category theory and algebra
developed by J.~Lurie~\cite{topos, HA, SAG}, see Chapter I.1 of~\cite{GR} for an excellent overview.

Throughout what follows, we will fix once and for all a symmetric
monoidal $(\oo, 2)$-category $\cA$ with unit object $1_\cA$. By forgetting non-invertible 2-morphisms we 
obtain a symmetric monoidal $(\oo,1)$-category $f(\cA)$, which we will abusively refer to as $\cA$
whenever only invertible higher morphisms are involved. Conversely, 
given a symmetric monoidal $(\oo, 1)$-category $\cC$, we can always
regard it as a symmetric monoidal $(\oo, 2)$-category $i(\cC)$ with
all $2$-morphisms invertible.\footnote{One can understand the above two operations as forming an adjoint pair $(i, f)$.}
Thus developments for higher $\oo$-categories equally
well apply to the more familiar $(\oo,1)$-categories. In what follows, noninvertible 2-morphisms only play a significant role starting with Section~\ref{dim functorial section}.

We will use $\otimes$ to denote the symmetric monoidal
structure of $\cA$.
We will
write $\Omega\cA= \End_\cA(1_\cA)$ for the ``based loops" in $\cA$, or in other words,
the symmetric monoidal $(\oo, 1)$-category of endomorphisms of
the monoidal unit $1_\cA$. Note that the monoidal unit $1_{\Omega\cA}$ is
nothing more than the identity $\id_{1_\cA}$ of the monoidal unit
$1_\cA$.  

\begin{example}[Algebras] Fix a symmetric monoidal $(\oo, 1)$-category $\cC$,
and let $\cA = \Alg(\cC)$ denote the Morita $(\oo,2)$-category of
algebras, bimodules, and intertwiners of bimodules within $\cC$.
The
forgetful map $\cA = \Alg(\cC) \to \cC$ is symmetric monoidal, and in
particular, the monoidal unit $1_\cA$ is the monoidal unit $1_\cC$
equipped with its natural algebra structure.
 Finally,
we have $\Omega\cA\simeq \cC$.  

For a specific example, one could take  $\cC =k\module=dgVect_k$ the $(\oo,
1)$-category of complexes of $k$-modules (with quasi-isomorphisms inverted). Then $\cA = \Alg(\cC)$ is the
$(\oo, 2)$-category of $k$-algebras, bimodules, and 
intertwiners of bimodules.
\end{example}

\begin{example}[Categories] A natural source of $(\oo,2)$-categories
is given by various theories of $(\oo,1)$-categories.  For example,  one could consider $\udg_k$, the $(\oo,2)$-category of
$k$-linear stable presentable $\oo$-categories (or $k$-linear presentable dg categories), $k$-linear
continuous functors, and natural transformations.

Observe that $\Alg(k\module)$ is a full subcategory of $\udg_k$,
via the functor assigning to a $k$-algebra its stable presentable
$\oo$-category of modules. The essential image
consists of dg categories admitting a  compact
generator.
\end{example}

%%%

\subsection{Dualizability}

\begin{defn}
An object $A$ of the symmetric monoidal $(\oo, 2)$-category $\cA$ is said to be {\em dualizable} (equivalently, $A$ is dualizable in the $(\oo,1)$-category
$f(\cA)$) if it admits a monoidal dual: there is a dual object $A^\vee \in \cA$
and evaluation and coevaluation morphisms
$$
\xymatrix{
\epsilon_A:A^\vee \otimes A\ar[r] & 1_\cA
&
\eta_A:1_\cA \ar[r] & A \otimes A^\vee
}$$
%satisfying the usual identities.
such that the usual compositions are naturally equivalent to the identity morphism
$$
\xymatrix{
A \ar[rr]^-{\eta_A \otimes \id_A} && A \otimes A^\vee \otimes A
\ar[rr]^-{ \id_A\otimes \epsilon_A} && A 
&
A^\vee \ar[rr]^-{ \id_{A^\vee} \otimes \eta_A} && A^\vee \otimes A \otimes A^\vee 
\ar[rr]^-{ \epsilon_A\otimes  \id_{A^\vee}} && A^\vee 
}
$$
\end{defn}

\begin{example}%[Algebra objects are dualizable]
Any algebra object $A\in \Alg(\cC)$ is {dualizable} with dual the opposite algebra $A^{op}\in \Alg(\cC)$.
The evaluation morphism 
$$
\xymatrix{
\epsilon_A: A^{op} \otimes A \ar[r] & 1_\cC
}
$$
is given by $A$ itself regarded as an $A$-bimodule.
The coevaluation morphism 
$$
\xymatrix{
\eta_A:  1_\cC \ar[r] &A \otimes A^{op} 
}
$$
is also given by $A$ itself regarded as an $A$-bimodule.
 
\end{example}

\subsubsection{Dualizable morphisms}\label{dualizable morphisms}

Consider two objects $A, B\in \cA$, and a morphism
$$
\xymatrix{
\Phi:A \ar[r] & B.
}$$

\begin{example}
If $\cA = \Alg(\cC)$, then $\Phi$ is simply an $A^{op}\otimes B$-module.
\end{example}

If $B$ is dualizable with dual $B^{\vee}$, we can package $\Phi$ in the 
equivalent form of the morphism $$
\xymatrix{
e_{\Phi}:B^{\vee} \otimes A \to 1_\cA}$$ defined by

$$\xymatrix{  
 B^\vee\ot A  \ar[d]_-{\id_{B^\vee} \ot \Phi} \ar[r]^-{e_\Phi} & 1_{\cA}\\
 B\ot B^\vee \ar[ur]^-{\epsilon_B} & }$$

If $A$ is dualizable with dual $A^{\vee}$, we can package $\Phi$ in the 
equivalent form of the morphism $$\xymatrix{
u_\Phi:1_\cA\to B\ot A^\vee}$$ defined by
$$\xymatrix{ & A\ot A^\vee\ar[d]^-{\Phi\otimes \id_{A^\vee}}\\
1_\cA\ar[ur]^-{\eta_A} \ar[r]^-{u_\Phi} & B\ot A^\vee}$$

If both $A$ and $B$ are dualizable, 
we can also encode $\Phi$ by its dual morphism 
$$
\xymatrix{
\Phi^\vee:B^\vee\ar[r] & A^\vee
}$$ 
defined by 
$$\xymatrix{ B^\vee  \ar@/_2pc/^-{\Phi^\vee}[rrrrrr]  \ar[rr]^-{\id_{B^\vee}\otimes \eta_A} && B^\vee \ot A\ot A^\vee 
\ar[rr]^-{\id_{B^\vee} \otimes \Phi\otimes \id_{A^\vee}}
&& B^\vee \ot B\ot A^\vee \ar[rr]^-{\epsilon_{B^\vee}\ot\id_{A^\vee}} && A^\vee}
$$
There is a natural composition identity 
$$
(\Phi\Psi)^\vee\simeq \Psi^\vee \Phi^\vee
$$
Note that for fixed $A, B$,  the construction $\Phi\mapsto \Phi^\vee$ naturally defines a {covariant} map
$$
\xymatrix{
(-)^\vee:\Hom(A,B)\ar[r]& \Hom(B^\vee,A^\vee)
}
$$ 
and in particular a morphism $\Phi_1\to \Phi_2$
induces a natural morphism $\Phi_1^\vee\to \Phi_2^\vee$.

Let us record the canonical equivalences encoded by the following commutative diagrams

\begin{equation}\label{standard identities}
\xymatrix{ & A\ot A^\vee\ar[d]^-{\Phi\otimes \id_{A^\vee}}                 &&&  
A^\vee\ot A \ar[dr]^-{\epsilon_{A^\vee}} &\\
1_\cA\ar[ur]^-{\eta_A} \ar[dr]_-{\eta_B} \ar[r]^-{u_\Phi} & B\ot A^\vee &&& 
B^\vee\ot A \ar[u]^-{\Phi^\vee\ot \id_A} \ar[r]^-{e_\Phi}\ar[d]_-{\id_{B^\vee} \ot \Phi} & 1_{\cA}\\
& B\ot B^\vee \ar[u]_-{\id_B\ot \Phi^\vee}                                                &&&
B\ot B^\vee  \ar[ur]_-{\epsilon_B} & }
\end{equation}

\begin{example} In the setting of algebras, bimodules and intertwiners, the morphisms $\Phi$, $u_\Phi$, $e_\Phi$ and $\Phi^\vee$ are all different manifestations of the same bimodule $\Phi$, making their various compatibilities particularly evident.
\end{example}

\begin{defn}
(1) A morphism $\Phi: A\to B$ is said to be {\em left dualizable}
 if it admits a left adjoint: there is a morphism $\Phi^\ell: B\to A$
and unit and counit morphisms
$$
\xymatrix{
\eta_ \Phi: \id_B \ar[r] & \Phi \circ \Phi ^\ell
&
\eps_\Phi : \Phi ^\ell \circ \Phi\ar[r] & \id_A
}$$
satisfying the usual identities.

(2) A morphism $\Phi: A\to B$ is said to be {\em right dualizable}
if it admits a right adjoint: there is a morphism $\Phi^r: B\to A$
and unit and counit morphisms
$$
\xymatrix{
\eta_ \Phi: \id_A \ar[r] & \Phi^r \circ \Phi 
&
\eps_\Phi : \Phi  \circ \Phi^r\ar[r] & \id_B
}$$
satisfying the usual identities. 

\end{defn}
%
%\begin{remark}
%Note that the composition monoidal structure of endomorphisms is not naturally symmetric.
%Hence  the notion of dualizability for endomorphisms is asymmetric: one could be careful and say that $\Phi^\vee$
%is the left dual of $\Phi$, and $\Phi$ is the right dual of $\Phi^\vee$.
%\end{remark}

\begin{remark}
If $A$ and $B$ are dualizable, and $\Phi: A\to B$ is left (resp.~right)  dualizable, then $\Phi^\vee: B^{\vee}\to A^{\vee}$ is right (resp.~left) dualizable with right adjoint 
$(\Phi^\ell)^{\vee}:A^{\vee}\to B^{\vee}$ (resp.~left adjoint $(\Phi^r)^{\vee}:A^{\vee}\to B^{\vee}$).
\end{remark}

%

%%%%%%%%%%%%%%%%%%%%%%%%%%%%%%%%%%%%%%%%%%%
%%%%%%%%%%%%%%%%%%%%%%%%%%%%%%%%%%%%%%%%%%%
%%%%%%%%%%%%%%%%%%%%%%%%%%%%%%%%%%%%%%%%%%%

\subsection{Traces and dimensions}
(We continue to refer to~\cite{TV, HSS} for thorough treatments of the theory of traces in higher category theory.)

Let $A\in \cA$ be a dualizable object with dual $A^{\vee}$.  Consider an
endomorphism
$$
\xymatrix{
\Phi:A \ar[r] & A
}$$
Since $A$ is dualizable, $\Phi$ has a trace defined as follows. 

\begin{defn}\label{trace1}

(1) The {\em trace} of $\Phi: A\to A$ is the object $\Tr(\Phi)\in
\Omega\cA$ defined by
$$\xymatrix{1_{\cA}\ar[r]^-{\eta_A}\ar@/_2pc/@{-}^-{\Tr(\Phi)}[rrrr]& A\ot A^\vee \ar[rr]^-{\Phi\ot\id_A}& &  A\ot A^\vee
\ar[r]^-{\epsilon_A}& 1_\cA}.$$

Given a natural transformation $\varphi:\Phi\to \Psi$, we define the induced morphism
$$
\xymatrix{
\Tr(\varphi):\Tr(\Phi)\ar[r] & \Tr(\Phi')
}
$$ by applying $\varphi\ot\id_{A^{\vee}}$ to the middle arrow above.

(2) The {\em dimension} (or \em{Hochschild homology}) of $A$ is the trace of the identity
$$
\dim(A)=\Tr(\id_A)\in \Omega\cA$$
or in other words, 
 the object defined by
$$\xymatrix{ 1_\cA\ar[rr]^-{\eta_A}\ar@/_2pc/@{-}^-{\dim(A)}[rrrr] &&
  A\ot A^\vee \ar[rr]^-{\epsilon_A}&& 1_\cA}$$

\end{defn}

\begin{remark}
Equivalently, we can describe the trace as the composition
$$
\xymatrix{1_{\cA}\ar[r]^-{\Phi}& \End(A)\ar[r]^-{\sim} &  A\ot A^{\vee}
\ar[r]^-{\epsilon_A}& 1_\cA}$$ where the middle arrow is the
identification deduced from the dualizability of $A$.
\end{remark}

\begin{remark}
Observe that  for fixed dualizable $A\in \cA$, taking traces gives a functor
$$
\xymatrix{
\Tr:\End(A) \ar[r] & \Omega\cA
}$$
\end{remark}

\begin{remark}
Observe that for any dualizable endomorphism $\Phi$,
the standard identities encoded by Diagrams~\ref{standard identities} give rise to an identification
$$
\Tr(\Phi)\simeq \Tr(\Phi^\vee)
$$ 
\end{remark}

\begin{example}
When $A = 1_\cA$ is the monoidal unit, and $\Phi:1_\cA \to 1_\cA$ is an endomorphism,
we have an evident equivalence of endomorphisms  
$$
\Tr(\Phi) \simeq \Phi
$$
\end{example}

\begin{thm}[\cite{TFT}]\label{thm lurie S1-action} There is a canonical   $S^1$-action on the dimension $\dim(A)$ of any
dualizable object $A$ of a symmetric monoidal $\oo$-category $\cA$.
\end{thm}

%%%

\subsubsection{Cyclic symmetry}

\begin{prop}\label{cyclic trace} 
Given two morphisms $$\xymatrix{A  \ar@<+.5ex>[r]^\Phi &
  \ar@<+.5ex>[l]^{\Psi}B}$$ 
between dualizable objects  $A, B\in \cA$, there is a canonical equivalence
$$
\xymatrix{
m(\Phi, \Psi):\Tr(\Phi\circ \Psi) \ar[r]^-\sim & \Tr(\Psi \circ \Phi)
}
$$ 
functorial in morphisms of  both $\Phi$ and $\Psi$.
\end{prop}

\begin{proof}

We construct $m(\Phi,\Psi)$ following the commutative diagram below:

$$\xymatrix{ & A\ot A^\vee\ar[r]^-{\Phi\otimes \id_{A^\vee}}     \ar[ddr]^-{\id_{A} \ot \Psi^\vee}       & B\ot A^\vee
 \ar[r]^-{\Psi\ot \id_{A^\vee}} \ar[ddr]^-{\id_{B} \ot \Psi^\vee}  &  A\ot A^\vee 
\ar[dr]^-{\epsilon_{A}} &\\
1_\cA\ar[ur]^-{\eta_A} \ar[dr]_-{\eta_B} & &  && 1_{\cA}\\
& B\ot B^\vee   \ar[r]_-{\Psi\ot \id_{B^\vee}}    & A\ot B^\vee \ar[r]_-{\Phi\ot \id_{B^\vee}} &
B\ot B^\vee  \ar[ur]_-{\epsilon_B} & }$$

Following the top edge, we find the definition of $\Tr(\Psi\circ\Phi)$. Following the bottom edge, we find the definition of $\Tr(\Phi\circ \Psi)$. The identifications filling the left and right diamonds arise from the standard identities encoded by Diagrams~\ref{standard identities}.  The identification filling the central square results from the symmetric monoidal structure.  

The construction is evidently functorial for  morphisms $\Phi\to \Phi'$.  The functoriality for  morphisms $\Psi\to \Psi'$ is  similar, once one recalls that the construction $\Psi\mapsto \Psi^\vee$ is covariantly functorial in morphisms of $\Psi$.
\end{proof}

\begin{example}
Taking $\Phi= \id_A$ yields a canonical equivalence 
$$
\xymatrix{
\gamma': \id_{\Tr(\Phi')} \ar[r]^-\sim & m( \id_A, \Phi'_A) 
}
$$
and likewise, taking
$\Phi'= \id_A$ yields a canonical equivalence 
$$
\xymatrix{
\gamma: \id_{\Tr(\Phi)} \ar[r]^-\sim & m( \Phi_A, \id_A) 
}
$$
Thus taking $\Phi= \Phi' =\id_A$ yields an automorphism of the identity of the Hochschild homology
$$
\xymatrix{
 (\gamma')^{-1}\circ \gamma:\id_{\Tr(\id_A)} \ar[r]^-\sim & \id_{\Tr(\id_A)}
}
$$
called the {\em BV homotopy}.
\end{example}
%Taking $\Phi' = \id_A$ gives an automorphism of $\Tr(\Phi)$.

\begin{remark} 
The proposition is only the initial part of the full cyclic symmetry of trace (see Remark \ref{full trace}), and the example is the lowest level structure of the
$S^1$-action on Hochschild homology (see Theorem~\ref{thm lurie S1-action}) defining  cyclic homology.
\end{remark}

%%%
%%%
%%%

\begin{lemma} \label{cyclic composition}
Given  morphisms $$\xymatrix{A\ar[r]^-{\Phi}& B\ar[r]^-{\Psi} & C\ar[r]^-{\Upsilon}& A}$$
between dualizable objects  $A, B, C\in \cA$, there is a canonical commutative diagram 
$$\xymatrix{\Tr(\Psi\Phi\Ups) \ar[rr]^-{m(\Psi,\Phi\Ups)} \ar[drr]_-{m(\Psi\Phi,\Ups)}&& \Tr(\Phi\Ups\Psi)\ar[d]^-{m(\Phi,\Ups\Psi)} \\ 
&& \Tr(\Ups\Psi\Phi)}$$

\end{lemma}

\begin{proof}

We construct the desired equivalence from the following diagram:

$$\xymatrix{ & C\ot C^\vee\ar[r]^-{\Ups}           \ar[dr]^-{\Psi^\vee}
& A\ot C^\vee  \ar[r]^-{\Phi} \ar[dr]^-{ \Psi^\vee}  
& B\ot C^\vee \ar[r]^-{\Psi}\ar[dr]^-{ \Psi^\vee} 
&  C\ot C^\vee 
\ar[dr]^-{\epsilon_{C}} &
\\
1_\cA\ar[ur] \ar[r] \ar[dr]
& B\ot B^\vee   \ar[r]^-{\Psi}  \ar[dr]^-{\Phi^\vee}  & C\ot B^\vee \ar[r]^-{\Ups} \ar[dr]^-{\Phi^\vee}&
A\ot B^\vee\ar[r]^-{\Phi} \ar[dr]^-{\Phi^\vee} & B\ot B^\vee \ar[r] & 1_{\cA}\\
& A\ot A^\vee\ar[r]^-{\Phi}          
& B\ot A^\vee  \ar[r]^-{\Psi}   
& C\ot A^\vee \ar[r]^-{\Ups}
&  A\ot A^\vee 
\ar[ur]^-{\epsilon_{A}} &
}$$

The natural transformations $m(\Psi,\Phi\Ups)$   and $m(\Phi,\Ups\Psi)$ describe passage from the top row
to the middle row and from the middle to the bottom, respectively. The transformation $m(\Psi\Phi,\Ups)$
can then be identified with the transformation from the top row to the bottom given by
inserting the diagonal morphisms $\id\ot \Phi^\vee\circ\Psi^\vee$ and using standard composition identities.
\end{proof}

\subsection{Functoriality of dimension}\label{dim functorial section}

Let $\cA^{cont} \subset \cA$ denote the $(\oo, 2)$-subcategory of dualizable objects and {\em continuous} or right
dualizable morphisms (morphisms that are left duals).

%
%\begin{remark}
%Suppose $\Psi: A\to B$ is a right dualizable morphism with right adjoint $\Psi^r:B\to A$.
%Then the
% unit and counit morphisms
%$$
%\xymatrix{
%\eta_ \Psi: \id_A \ar[r] & \Psi^r \circ \Psi 
%&
%\eps_\Psi : \Psi  \circ \Psi^r\ar[r] & \id_B
%}$$
%induce morphisms of traces
%$$
%\xymatrix{
%\eta_ \Psi: \Tr(\id_A) \ar[r] & \Tr(\Psi^r \circ \Psi)
%&
%\eps_\Psi : \Tr(\Psi  \circ \Psi^r)\ar[r] & \Tr(\id_B)
%}
%$$
%\end{remark}
%

\begin{defn}\label{functorial dim}
Let $\Psi:A\to B$ denote a morphism in $\cA^{cont}$ with right adjoint
$\Psi^r:B\to A$.  We define the induced morphism of dimensions
$$
\xymatrix{
\dim(\Psi): \dim(A) \ar[r] & \dim(B)
}
$$
to be the composition 

$$
\xymatrix{
\Tr(\id_A) \ar[r]^-{\eta_{\Psi}} &
\Tr(\Psi^r\circ \Psi)
\ar[rr]^-{m(\Psi^r, \Psi)} &&
\Tr( \Psi\circ \Psi^r) 
 \ar[r]^-{\eps_{\Psi}} &
\Tr(\id_B)
}
$$
\end{defn}

\begin{remark}\label{rem functorial dim}
In other words, the morphism $\dim(\Psi)$ is defined by the following diagram
$$\xymatrix{
&                & A\otimes A^\vee \ar@<-.5ex>[dd]_-{\Psi\ot \id_{A^\vee}}\ar[rrdd]^-{\epsilon_A}& &               &&&&&  \\
 & &  & &                                                                                                                          &&&&& \\
1_\cA\ar[rruu]^-{\eta_A}\ar[rr]^-{u_\Psi}\ar[rrdd]_-{\eta_B}&& 
B\otimes A^\vee  \ar@<-.5ex>[uu]_-{\Psi^r\ot\id_A} \ar[rr]^-{c_\Psi}\ar@<+.5ex>[dd]^-{\id_{B^\vee}\ot\Psi^{r\vee}}&&1_\cA    & 1_{\cA}\ar[rrrr]^-{\Tr(\Psi^r\Psi)\simeq \Tr(\Psi\Psi^r)}
\ar@/_4pc/@{-}[rrrr]^{\dim(B)} \ar@/^4pc/@{-}[rrrr]^{\dim(A)}&&&& 1_\cA \\      
 & &  & &                                      &&&&&                                                                                                               \\
 &                & B \otimes B^\vee\ar@<+.5ex>[uu]^-{\id_B\ot \Psi^\vee} \ar[rruu]_-{\epsilon_B}& &          &&&&&                                              
}
$$

Following the top and bottom edge, we find  the respective definitions of $\dim(A)$ and $\dim(B)$.
The unit $\eta_\Psi$ defines a morphism from the top edge to the top zig-zag. 
 The counit $\epsilon_\Psi$ defines a morphism from the bottom zig-zag to the bottom edge.
The passage from the top to bottom zig-zag is given by the construction $m(\Psi^r,\Psi)$ and the identification $$\Tr(\Psi^{r\vee}\circ\Psi^\vee)\simeq \Tr((\Psi\circ\Psi^r)^\vee)\simeq\Tr(\Psi\circ\Psi^r)$$ 
\end{remark}

\begin{prop}\label{tracefunctoriality}
For a diagram
$$\xymatrix{ A\ar[r]^-{\Phi} & B\ar[r]^-{\Psi} & C}$$ 
within
$\cA^{cont}$, there is a canonical equivalence
$$ \xymatrix{ \dim(\Psi\circ \Phi) \simeq \dim(\Psi) \circ \dim(\Phi)
  : \dim(A) \ar[r] & \dim(C) }
$$
\end{prop}

\begin{proof}
The equivalence is given by filling in the following diagram
$$\xymatrix{
\dim(A) \ar[r]^-{\eta_{\Phi}} \ar[dr]^-{\eta_{\Psi}} &  \Tr(\Phi^r\Phi)  \ar[r]^-{m} \ar[d]^-{\eta_{\Psi}}  & \Tr(\Phi\Phi^r) \ar[r]^-{\epsilon_{\Phi}} \ar[d]^-{\eta_{\Psi}} &
\dim(B) \ar[d]^-{\eta_{\Psi}} \\
&\Tr(\Phi^r \Psi^r \Psi\Phi)  \ar[r]^-{m} \ar[dr]^-{m}&
\Tr(\Psi^r \Psi  \Phi \Phi^r) \ar[r]^-{\epsilon_{\Phi}} \ar[d]^-{m}& \Tr(\Psi^r \Psi ) \ar[d]^-{m}\\
&&\Tr(\Psi\Phi\Phi^r\Psi^r) \ar[r]^-{\epsilon_{\Phi}} \ar[dr]^-{\epsilon_{\Psi\Phi}}&\Tr(\Psi\Psi^r) \ar[d]^-{\epsilon_{\Psi}}\\
&&&\dim(C) 
}
$$

Along the three boundary edges, we find the definitions of $\dim(\Phi)$, $\dim(\Psi)$ and $\dim(\Psi\Phi)$ respectively. 

The two corner triangles are given by the composition identities for adjoints (for example, at  the top left, relating the adjoint of $\Phi\Psi$ with the composition of adjoints of $\Psi$ and $\Phi$).

The middle triangle is given by the identity of Lemma \ref{cyclic composition}.

The top right square is given by taking traces of the evident commutative diagram of endomorphisms
$$\xymatrix{
\Phi\Phi^r \ot \Id_{B^\vee}  \ar[d] \ar[r]&  \Id_B \ot \Id_{B^\vee} \ar[d]\\
\Phi\Phi^r \ot (\Psi^r\Psi)^\vee \ar[r]& \Id_B\ot (\Psi^r\Psi)^\vee 
}
$$
and using the canonical identification $\Tr(F)=\Tr(F^\vee)$ for any dualizable morphism.

Finally, the two remaining commuting squares are given by the functoriality of the cyclic
rotation of the trace in its two arguments. For instance, in the top left square, we may either
rotate $\Tr( \Phi^r\circ (\Id_A\circ \Phi))$ and then apply the unit $\eta_\Psi:\Id_A\to \Psi^r\Psi$
or first apply the unit and then rotate.

This concludes the construction.
 \end{proof}

%\subsection{Grothendieck-Riemann-Roch}

Since we have an evident equivalence
$\dim(1_\cA)\simeq 1_\cA$ for the unit $1_\cA\in \cA$, we have the following specialization of
Proposition~\ref{tracefunctoriality} in which we  adopt suggestive
notation.

\begin{corollary}[Abstract Grothendieck-Riemann-Roch]
Let $A, B\in \cA^{cont}$ and $V:1_\cA\to A$ and $\pi_*:A\to B$ morphisms in
$\cA^{cont}$. Then the following diagram naturally commutes
$$ \xymatrix{ \ar[dr]_-{\dim(\pi_* V)} 1_\cA
  \ar[r]^-{\dim(V)} & \dim(A) \ar[d]^-{\dim(\pi_*)}\\ &
  \dim(B) }
$$

\end{corollary}

\begin{remark}
One can show along the same lines as the proposition that taking dimensions extends to a symmetric monoidal functor
$$ \xymatrix{ \dim:\cA^{cont}\ar[r] & \Omega\cA.  }$$ 
\end{remark}

\subsection{Functoriality of traces}

We would like to capture the functoriality for traces of arbitrary endomorphisms of
dualizable objects.  For this purpose we define a morphism between pairs 
$$
\xymatrix{A\in \cA^{cont} & \Phi_A\in \End_{\cA}(A)}
$$
of an object and an endomorphism to consist of a pair 
$$\xymatrix{ \Psi\in \Hom_{\cA^{cont}}(A,B) &
  \psi:\Psi\circ\Phi_A\ar[r]^-{\simeq} & \Phi_B\circ \Psi}
  $$ 
of a morphism and a {\em commuting structure}.

\begin{defn}\label{functorial trace}
For a morphism $$(\Psi,\psi):(A,\Phi_A)\to (B,\Phi_B)$$ as above, we define the induced morphism of traces
$$ \xymatrix{ \Tr(\Psi,\psi): \Tr(\Phi_A) \ar[r] & \Tr(\Phi_B) }
$$
to be the composition
$$
\xymatrix{
\Tr(\Phi_A) \ar[r]^-{\eta_{\Psi}} &
\Tr(\Psi^r \Psi \Phi_A) \ar[r]^-{\psi}
  & \Tr(\Psi^r \Phi_B  \Psi)
\ar[rr]^-{m(\Psi^r, \Phi_B\Psi)} &&
\Tr( \Phi_B \Psi \Psi^r) 
 \ar[r]^-{\epsilon_\Psi} &
\Tr(\Phi_B)
}
$$
\end{defn}

\begin{remark}
Note that we could alternatively define a morphism $\Tr(\Psi,\psi)$ by
applying the unit $\eta_\Psi$ to the right of $\Phi_A$, rotating the
trace in the opposite direction, and again using the counit on the
right. It is elementary to give a natural equivalence of the two
constructions using nothing more than the dualizability of $A$. 
\end{remark}
%
%\begin{example}
% When $A=B$, 
%$\Psi = \id_A$, and $\psi = \id_{\Phi}$, the
%equivalence is the identity.
%\end{example}

\begin{remark}In parallel with Remark~\ref{rem functorial dim} about the functoriality of dimensions, it is enlightening to realize the functoriality of traces as a chase through the following diagram
$$
\xymatrix{
 &                & A \otimes A^\vee\ar@<-.5ex>[dd]_-{\Psi\ot\id_{A^\vee}} \ar[rr]^-{\Phi_A\ot \id_{A^\vee}}&
& A\otimes A^\vee\ar[ddrr]^-{\epsilon_A} \ar@<-.5ex>[dd]_-{\Psi\ot\id_{A^\vee}} & &                                                           &&&&&&       \\
 && & &
 & &                                                                                                                              &&&&&&  \\
1_\cA \ar[uurr]^-{\eta_A} \ar[rr]^-{u_\Psi}\ar[ddrr]_-{\eta_B}
&& 
B\otimes A^\vee \ar@<-.5ex>[uu]_-{\Psi^r\ot\id_{A^\vee}} \ar[rr]^-{\Phi_B\ot\id_{A^\vee}}\ar@<+.5ex>[dd]^-{\id_{B^\vee}\ot \Psi^{r\vee}}&& B\otimes A^\vee \ar[rr]^-{c_\Psi}\ar@<+.5ex>[dd]^-{\id_{B^\vee}\ot \Psi^{r\vee}}\ar@<-.5ex>[uu]_-{\Psi^r\ot\id_{A^\vee}} &
&1_{\cA}                                                                                                                           & 1_{\cA} \ar[rrrrr]^-{\Tr(\Psi^r \Psi \Phi_A)\simeq}_-{\Tr(\Psi^r \Phi_B  \Psi)\simeq
\Tr( \Phi_B \Psi \Psi^r)} 
\ar@/_5pc/@{-}[rrrrr]^{\Tr(\Phi_B)} \ar@/^5pc/@{-}[rrrrr]^{\Tr(\Phi_A)}&&&&& 1_{\cA}     \\
 &&  & 
&
& &                                                                                                                               &&&&&& \\
 &                & B \otimes B^\vee\ar@<+.5ex>[uu]^-{\id_B\ot\Psi^\vee} \ar[rr]^-{\Phi_B\ot \id_{B^\vee}} & 
& B\otimes B^\vee\ar@<+.5ex>[uu]^-{\id_B\ot\Psi^\vee} \ar[uurr]_-{\epsilon_B}& &                                                                 &&&&&&
}
$$
\end{remark}

\begin{prop}
Suppose given objects $A,B,C\in \cA^{cont}$, endomorphisms $\Phi_A,\Phi_B,\Phi_C$, a  commutative diagram of continuous morphisms
$$
\xymatrix{ A\ar@/^2pc/[rrrr]^-{\Psi_{AC}}\ar[rr]^-{\Psi_{AB}} && B\ar[rr]^-{\Psi_{BC}} && C}$$ 
and commuting structures 
$$
\xymatrix{
s_{AB}:\Psi_{AB}\Phi_A\ar[r]^-\sim & \Phi_B \Psi_{AB} &
s_{BC}:\Psi_{BC}\Phi_B\ar[r]^-\sim & \Phi_C \Psi_{BC} &
s_{AC}:\Psi_{AC}\Phi_A\ar[r]^-\sim & \Phi_C \Psi_{AC}
}$$ 
with an identification $s_{AC}\simeq s_{BC} s_{AB}$.
Then there is a canonical equivalence
$$ \xymatrix{ \Tr(\Psi_{AC},s_{AC}) \simeq \Tr(\Psi_{BC},s_{BC}) \circ \Tr(\Psi_{AB},s_{AB})
  : \Tr(\Phi_A) \ar[r] & \Tr(\Phi_C) }
$$
\end{prop}

\begin{proof}
The construction is obtained from following a minor expansion of the diagram proving 
Proposition \ref{tracefunctoriality}. The additional moves needed are commuting the commuting structures past
the symmetry $m$ of the trace and the unit and counits of the adjunctions. These all follow immediately from the 2-categorical interchange law
for natural transformations.
\end{proof}

\begin{remark}\label{full trace}
The full functoriality of the trace $\Tr$ takes roughly the following form, see~\cite{HSS} for a detailed treatment.
Define
the {\em loop category} $\cL^{cont}\cA$ to be the symmetric monoidal
$\oo$-category with objects consisting of pairs $(A,\Phi_A)$ of a dualizable object $A\in \cA$ equipped with a (not necessarily continuous) endomorphism $\Phi_A$, and 
morphisms given by pairs $(\Psi,\psi)$ as above with $\Psi$ continuous. 
Taking traces to extend to a symmetric monoidal functor
$$ \xymatrix{ \Tr:\cL^{cont}\cA\ar[r] & \Omega\cA }$$ 
extending the dimension functor 
$$
\xymatrix{
\dim:\cA^{cont} \ar[r] & \Omega \cA
}$$
for constant loops $\Phi_A= \id_A$, and trivial commuting structures $\psi = \id_{\Psi}$.
%so that the resulting diagram of symmetric monoidal functors naturally
%commutes:
%$$ \xymatrix{ \ar[dr]_-{\dim} \cA^{cont}
%  \ar[r]^-{\Delta} & \cL\cA^{cont} \ar[d]^-{\Tr}\\ &
%  \Omega\cA^{cont} }
%$$

In order to capture the full cyclic symmetry of the trace $\Tr$, one should further extend it to
a homotopical trace valued in $\Omega \cA$, or in other words, to the appropriate
full cyclic bar construction  (of which the above forms only the one-simplices). 
\end{remark}
.

%% \begin{proof}
%% Consider the two endomorphisms $\Phi\circ\Psi$ and $\Psi\circ \Phi$ of
%% $A$. We have two identifications of the resulting traces:
%% $$\xymatrix{\Tr(\Phi\circ\Psi) \ar@/^2pc/@{-}^-{m(\Phi,\Psi)}[rr]
%% \ar@/_2pc/@{-}^-{\Tr(\alpha)}[rr] && \Tr(\Psi\circ \Phi)}.$$

%% \end{proof}

\section{Traces in Geometry}\label{sect geom}

\subsection{Categories of correspondences}
For concreteness, we fix a base commutative ring, and  work in the symmetric monoidal $(\oo,1)$-category $Stacks_k$ of derived stacks over $\Spec k$. 
It is worth pointing out that the constructions of
this section apply in any presentable $\oo$-category with the Cartesian symmetric monoidal structure.

Let $Corr_k$ denote the symmetric monoidal $\oo$-category of correspondences in $Stacks_k$. Thus
morphisms are given by the classifying space of  correspondences 
$$
\xymatrix{X&\ar[l]\ar[r] Z & Y}
$$ 
so all higher
morphisms are isomorphisms. Composition of correspondences is given
by the derived fiber product. The based loop category 
$$\Omega Corr_k=\End_{Corr_k}(\Spec k)\simeq Stacks_k$$
is again derived stacks, regarded as self-correspondences of the point $\Spec k$.

We will also enhance $Corr_k$ to  the  symmetric monoidal $(\oo,2)$-category $\ul{Corr}_k$  where we now allow noninvertible maps of correspondences
$$
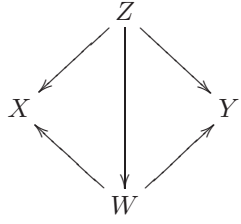
\xymatrix{& Z \ar[ld]\ar[rd] \ar[dd] &\\
X & & Y\\
& W\ar[ul]\ar[ur] &}$$
In other words,
the morphisms $\ul{Corr}_k(X,Y)$ now form the $\oo$-category $Stacks_{/X\times Y}$ of stacks over $X\times Y$ with arbitrary morphisms
rather than isomorphisms as in $Corr_k(X,Y)$.

We will also have need to restrict the class of morphisms of correspondences to some subcategory
of $Stacks_{/X\times Y}$.  In particular, we will consider the subcategory $\ul{Corr}_k^{prop}$ 
in which we only allow proper maps of correspondences.

%%%

\subsection{Traces of correspondences}

Given a map $Z\to X$, it is convenient to introduce the symmetric presentation of the  based loop space 
$$\cL_Z X
= Z\times_{Z\times X} Z
$$
Note the two natural identification with the traditional based loop space
$$
\xymatrix{
  \cL X \times_X Z \simeq X \times_{X \times X} Z & \ar[l]_-\sim  Z\times_{Z\times X} Z  \ar[r]^-\sim & Z \times_{X \times X} X \simeq Z \times_X \cL X
}$$
There is a natural rotational equivalence $\cL X \times_X Z \simeq Z \times_X \cL X$  that makes the above two identifications coincide. (It does not preserve base points and is not given by swapping the factors). Thus we can unambiguously identify all of the above versions of the based loop space.

\begin{prop} \label{basic stack dimensions}

(1) 
Any derived stack $X$ is dualizable as an object of $Corr_k$, with dual $X^\vee$ identified with $X$ itself, and dimension $\dim(X)$ identified 
with the loop space 
$$
\cL X= X^{S^1}\simeq X\times_{X\times X} X
$$  regarded as a self-correspondence of $pt=\Spec k$.

(2) The transpose of any correspondence $X\leftarrow Z\to  Y$ is identified with the reverse correspondence
$Y \leftarrow Z \to  X$. The trace of  a self-correspondence $X\leftarrow Z\to X$
is identified with the based loop space
$$\xymatrix{ 
\Tr(Z)\simeq Z|_{\Delta_X} = Z\times_{X\times X} X \simeq \cL_Z X
 }$$ 
  regarded as a self-correspondence of $pt=\Spec k$.
  
  In particular, the trace of the graph $\Gamma_f \to X\times X$ of a self-map $f:X\to X$ is
identified with the fixed point locus 
$$
\xymatrix{
\Tr(f) \simeq  \Gamma_f|_{\Delta_X} = \Gamma_f \times_{X\times X} X \simeq X^f  
}$$

\end{prop}

\begin{proof}
The evaluation and coevaluation presenting the self-duality of $X$ are both
given by $X$ itself as a correspondence between $pt=\Spec k$ and
$X\times X$ via the diagonal map. The standard identities follow
from  the calculation of the fiber product of the two diagonal maps
$$
X{}_{\Delta_{12}}\times_{X\times X\times X}{}_{\Delta_{23}} X\simeq X
$$
Thus the dimension of $X$ is  the loop space
$$
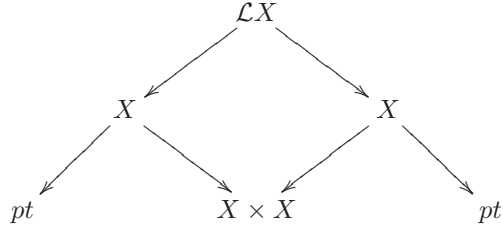
\xymatrix{&&\cL X \ar[ld]\ar[rd]&&\\
&X\ar[ld]\ar[rd]&&X\ar[ld]\ar[rd]& \\
pt&&X \times X&&pt}$$

By definition, the transpose of a correspondence $X\leftarrow Z\to  Y$ is identified with
 $Y\leftarrow Z\to  X$  by checking the definition
$$
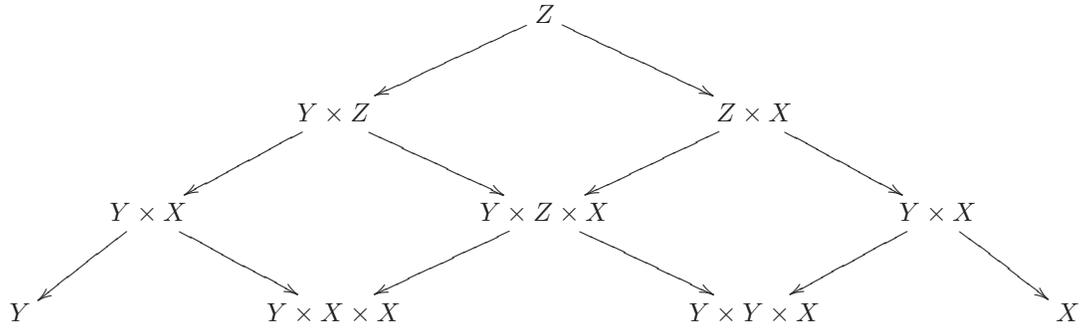
\xymatrix{&&& Z \ar[ld]\ar[rd] &&&\\
&& Y \times Z \ar[ld]\ar[rd]   && Z  \ar[ld]\ar[rd] \times X && \\
&Y \times X\ar[ld]\ar[rd]&& Y \times Z \times X \ar[ld]\ar[rd] && Y \times X\ar[ld]\ar[rd]& \\
Y&&Y \times X \times X&&Y \times Y \times X && X}$$
The trace of a self-correspondence $X\leftarrow Z\to  X$ is then calculated by the composition 
$$
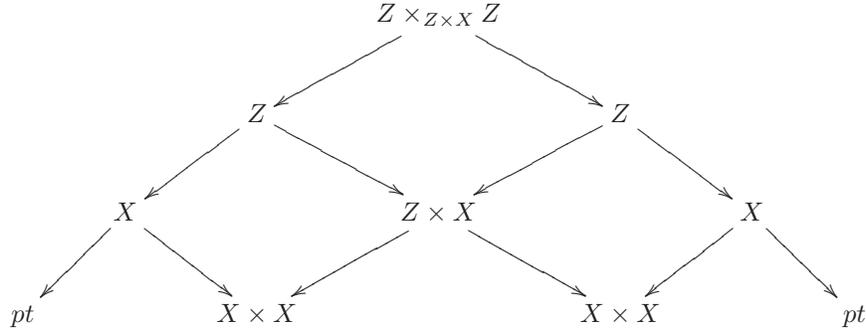
\xymatrix{&&& Z\times_{Z\times X} Z \ar[ld]\ar[rd] &&&\\
&&  Z\ar[ld]\ar[rd]   &&   \ar[ld]\ar[rd] Z  && \\
&X \ar[ld]\ar[rd]&&Z \times X \ar[ld]\ar[rd] &&  X\ar[ld]\ar[rd]& \\
pt&&X \times X &&X \times X && pt}$$

%$$\xymatrix{ Y\ar[rr]^-{\id_Y\times X}&& Y\times X\times X \ar[rr]^-{\id_Y\times Z\times \id_X} && Y\times Y\times X\ar[rr]^{Y\times \id_X} && X}$$

Finally, the case of the graph $Z=\Gamma_f$  of a  self-map
gives the fixed point locus by definition.
\end{proof}

\begin{remark}[Cyclic version]\label{rem: cyclic}
The identification $\dim(X)\simeq \cL X$ above is naturally $S^1$-equivariant for the standard loop
rotation on $\cL X$ and the cyclic symmetry of $\dim(X)$ provided by the cobordism hypothesis. 
To see this it is useful to consider $X$ as an $E_\infty$-algebra object in
$Stacks_k^{op}$ via the diagonal map (or as an $E_n$-object for any $n$).
In other words, for $n=1$ we identify stacks and correspondences with objects and morphisms in the Morita category $Alg(Stacks_k^{op})$.
It follows from the properties of topological chiral homology \cite[Theorem 5.3.3.8]{HA} that for a (constant) commutative algebra $A$
its topological chiral homology over a manifold is given by the tensoring of commutative algebras
over simplicial sets $\int_M A= M\ot A$. In particular (passing back from the opposite category to stacks)
we have $\int_{S^1} X = X^{S^1}=\cL X$. Moreover this identification holds not just for a fixed circle but over the moduli space of circles $BDiff(S^1)\sim BS^1$, i.e. equivariantly for rotation. We also know
from \cite[Example 5.3.3.14]{HA} or \cite[Example 4.2.2]{TFT} that the $S^1$-action on the dimension of an associative algebra $A$ (given classically by the cyclic structure on the Hochschild chain complex) is given by the rotation $S^1$-action on the topological chiral homology $\int_{S^1} A$, i.e., the family of topological chiral homologies over the moduli space of circles. In our case this recovers the rotation action on the loop space.
\end{remark}

\subsection{Geometric functoriality of dimension}

\begin{prop}\label{geometric adjunction} The graph $X\leftarrow \Gamma_f \to Y$ of any proper morphism  
$f: X\to Y$ gives a continuous  morphism $F:X\to Y$ in $\ul{Corr}^{prop}_k$, with right adjoint $F^r:Y\to X$ identified with the opposite correspondence $Y\leftarrow \Gamma_f \to X$. 
\end{prop}

\begin{proof}
We construct the unit and counit of the adjunction as follows. Consider the composition $F^r F:X\to X$ of correspondences 
$$
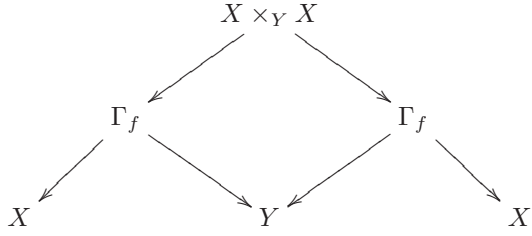
\xymatrix{&&X \times_Y X \ar[ld]\ar[rd]&&\\
&\Gamma_f\ar[ld]\ar[rd]&&\Gamma_f\ar[ld]\ar[rd]& \\
X&&Y &&X}$$
The unit
$\eta_f: \id_X=X \to F^rF\simeq X\times_Y X$ is given by the relative diagonal map.

 Consider the opposite composition of correspondences 
$$
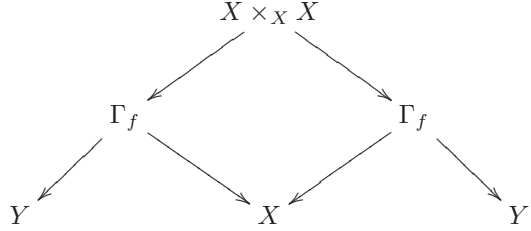
\xymatrix{&&X \times_X X \ar[ld]\ar[rd]&&\\
&\Gamma_f\ar[ld]\ar[rd]&&\Gamma_f\ar[ld]\ar[rd]& \\
Y&&X &&Y}$$
The counit $\epsilon_f:  FF^r\simeq X\to \id_Y=Y$
is given by $f$ itself. 

The standard identities are easily verified by identifying the resulting composite map
$$\xymatrix{\Gamma_f\ar[r] & \Gamma_f\times_Y \Gamma_f \times_X \Gamma_f \ar[r] & \Gamma_f}$$ 
 of correspondences with the identity.
%%$$\xymatrix{Y\ar[r] & Y\times X\ar[r] & Y}.$$
\end{proof}

\begin{lemma}\label{geometric cyclic trace}
Let $F_Z:X\to Y$ and $F_W:Y\to X$ be morphisms in $Corr_k$ given by respective correspondences $X\leftarrow Z \to Y$ and $Y \leftarrow W \to X$. Then the canonical equivalence 
$$
\xymatrix{
m(F_W, F_Z):\Tr(F_W\circ F_Z) \ar[r]^-\sim & \Tr(F_Z\circ F_W)
}
$$ 
is given by the composition of evident geometric identifications
$$
\xymatrix{
  (Z \times_Y W) \times_{X \times X} X \ar[r]^-\sim &  W \times_{X\times Y} Z \ar[r]^-\sim & Z \times_{Y\times X} W
  \ar[r]^-\sim &  (W \times_X Z) \times_{Y \times Y} Y 
}
$$
\end{lemma}

\begin{proof} Returning to the definition and using our previous identifications, observe that $m(F_Z, F_W)$ is calculated by commutativity of the diagram of correspondences
$$\xymatrix{ & X\times X\ar[r]^-{Z \times X}     \ar[ddr]^-{X \times W}      & Y\times X
 \ar[r]^-{W\times X} \ar[ddr]^-{Y \times W}  & X\times X
\ar[dr]^-{X} &\\
pt\ar[ur]^-{X} \ar[dr]_-{Y} & &  && pt\\
& Y\times Y   \ar[r]_-{W\times Y}    & X\times Y \ar[r]_-{X \times Y} &
Y\times Y  \ar[ur]_-{Y} & }$$

Following the top edge, we see $\Tr(F_W\circ F_Z) \simeq  (Z \times_Y W) \times_{X \times X} X$.
Following the bottom edge, we see $\Tr(F_Z\circ F_W) \simeq  (W \times_X Z) \times_{Y \times Y} Y$.
Moving from the top to bottom edge via the successive equivalences of the three commuting squares,
one finds the three successive equivalences in the assertion of the lemma.
\end{proof}

\begin{prop}\label{geometric dims} 
Suppose $f:X\to Y$ is a proper 
morphism, and  $F:X\to Y$ denotes the induced morphism  in $\ul{Corr}^{prop}_k$ given by the graph $X\leftarrow \Gamma_f \to Y$.
Then $\dim(F):\dim(X)\to\dim(Y)$ is canonically identified 
with the $S^1$-equivariant morphism $\cL f:\cL X\to \cL Y$.
\end{prop}

\begin{proof}
Denote by $F^r:Y\to X$ the right adjoint to $F$. We must calculate 
$$
\xymatrix{
\dim(X) \ar[r] & \Tr(F^r F) \ar[rr]^-{m(F^r, F)} && \Tr(F F^r) \ar[r] & \dim(Y)
}$$
We have seen that the first and third morphisms correspond to the natural geometric maps
$$
\xymatrix{
\cL X \simeq X\times_{X \times X} X \ar[r] & (X \times_Y X) \times_{X \times X} X &
X \times_{Y \times Y} Y \ar[r] & Y \times_{Y \times Y} Y \simeq \cL Y
}
$$
induced by the relative diagonal $X\to X\times_Y X$ and given map $f:X\to Y$ respectively.
Furthermore, by Lemma~\ref{geometric cyclic trace}, the middle map is the natural geometric identification
$$
\xymatrix{
(X \times_Y X) \times_{X \times X} X \ar[r]^-\sim & 
%X \times_{Y \times X} X \ar[r]^-\sim &
X \times_{Y \times Y} Y}
$$

Altogether, the  composition
%$$
%\xymatrix{
%\cL X\ar[r] &  \cL_X Y \ar[r] & \cL Y
%}$$ 
is easily identified with the loop map $\cL f:\cL X \to \cL Y$.
\end{proof}

\begin{remark}
%Recall that the composition $F^r F$ can also be identified with the
%correspondence $$\xymatrix{X&\ar[r]\ar[l] X\times_Y X& X}$$ 
%and
%identify its trace as the restriction to the diagonal,
%$$\cL_Y X\simeq (X\times_Y X)\times_{X\times X} X.$$ Likewise the
%composition $f_*f^!$ is the correspondence $$\xymatrix{Y&\ar[r]\ar[l]
%  X& Y},$$ with trace its restriction to the diagonal,
%$$\cL_Y X\simeq X\times_{Y\times Y} Y.$$
%
It follows from the proposition that the loop map $\cL f:\cL X \to \cL Y$ must be proper when the given map $f:X\to Y$ is proper.
Let us note why this is true geometrically from the factorization $\cL X \to \cL_X Y \to \cL Y$ appearing in the proof. 

First, the natural morphism $\cL X\to \cL_X Y$ is the restriction
along the diagonal $X\to X\times X$ of the relative diagonal $X\to \XYX$. The relative diagonal is a closed
embedding since $f$ is proper, and hence the natural morphism $\cL X\to \cL_X Y$ is as well. 
Second, the natural  morphism $\cL_X Y
\to \cL Y$ is the restriction along the diagonal $Y\to Y\times Y$ of the proper
morphism $f:X\to Y$ and thus is proper as well. 
Altogether, we see that $\cL f:\cL X\to \cL Y$ is itself
proper.
\end{remark}

\begin{remark} 
One can invoke the cobordism hypothesis with singularities to endow
the morphism $\dim(F):\dim(X)\to \dim(Y)$ with a canonical $S^1$-equivariant structure, and it will agree with the canonical geometric 
$S^1$-equivariant structure on the map $\cL f:\cL X\to \cL Y$  under the identification of the proposition.
\end{remark}

%\subsection{Characters and fixed point loci}
%We now spell out a special case of the loop space formalism. Namely suppose an algebraic group $G$ acts on a space $X$,
%with quotient $X/G$, 
%......

\subsection{Geometric functoriality of trace}

Consider a proper morphism $f:X\to Y$ and endomorphisms 
 $F_Z:X\to X$ and $F_W:Y\to Y$ in $Corr_k$ given by respective self-correspondences 
 $X\leftarrow Z \to X$ and $Y \leftarrow W \to Y$. 

  By an $f$-morphism from the pair $(X,F_Z)$ to the pair $(Y,F_W)$, we mean an
  identification 
  $$
  \xymatrix{
  s:Z  \ar[r]^-\sim &  X\times_Y W
  }$$ 
  of correspondences from $X$ to $Y$. This in turn
  induces an identification of what might be called relative traces
  $$
  \xymatrix{
  Z\times_{Y\times Y} Y\ar[r]^-\sim &  X\times_{Y\times Y} W
  }
  $$ 
  generalizing the relative loop space $\cL_X Y$ from the case of the identity correspondences $Z=X$, $W=Y$.
 We thus obtain a map of traces
$$
\xymatrix{
\tau(f, s):Z|_{\Delta_X}  =   Z\times_{X\times X} X \ar[r] &   Z \times_{Y\times Y} Y \ar[r]^-\sim &  X \times_{Y\times Y} W \ar[r] &  
  Y\times_{Y\times Y} W = W|_{\Delta_Y}
}
$$

\begin{prop}\label{geometric traces} 
With the preceding setup, the trace map $\Tr(f,s):\Tr(F_Z)\to\Tr(F_W)$ is
canonically identified with the geometric map 
$$
\xymatrix{
\tau(f, s): Z|_{\Delta_X} \ar[r] & W|_{\Delta_Y}
}$$
\end{prop}

\begin{proof}
Denote by $F:X\to Y$  the morphism given  by the graph $X\leftarrow \Gamma_f \to Y$, and by $F^r:Y\to X$ its right adjoint. 
We must calculate 
$$
\xymatrix{
\Tr(F_Z) \ar[r] & \Tr(F^r F F_Z) \ar[r]^-{s } &  \Tr(F^r F_W F)   \ar[rr]^-{m(F^r, F_W F)} && \Tr(F_W F F^r) \ar[r] & \Tr(F_W)
}$$

We have seen that the first and fourth morphisms correspond to the natural geometric maps
$$
\xymatrix{
Z|_{\Delta_X} = Z\times_{X \times X} X \ar[r] & Z  \times_{X \times X} (X\times_Y X)  % \simeq Z\times_{Y \times Y} Y 
}
$$
$$\xymatrix{
X \times_{Y \times Y} W \ar[r] & Y \times_{Y \times Y} W = W|_{\Delta_Y}
}
$$
induced by the relative diagonal $X\to X\times_Y X$ and given map $f:X\to Y$ respectively.
%Note that these are the initial and final map in the composition $\tau(f, s)$ respectively.

Using associativity, the second map, induced by $s$, is  the natural geometric identification
$$
\xymatrix{
Z  \times_{X \times X} (X\times_Y X)  \simeq Z \times_{Y\times Y} Y\ar[r]^-\sim &  W \times_{Y \times Y} X 
%Z\times_{Y \times Y} Y  \ar[r]^-\sim & X\times_{Y\times Y} W
}
$$
By Lemma~\ref{geometric cyclic trace}, the third map, given by the cyclic symmetry, is nothing more than the natural identification
$$
\xymatrix{
W \times_{Y \times Y} X  \ar[r]^-\sim &
X \times_{Y \times Y} W
}
$$

Thus assembling the above maps we arrive at the composition defining $\tau(f, s)$.
\end{proof}

\section{Traces for sheaves}\label{sect shvs}

In this section, we spell out how to apply the abstract formalism of traces of Section~\ref{sect cat}
and its geometric incarnation of Section~\ref{sect geom} to categories of sheaves. Recall Definition~\ref{sheaf def intro}:

\begin{defn} A {\em sheaf theory} is a symmetric monoidal functor of $(\infty,2)$-categories
$$
\xymatrix{
\ul{\cS}:\ul{Corr}_k^{prop}\ar[r] &  \ul{dgCat}_k
}
$$ 
from correspondences of stacks (with 2-morphisms given by proper maps of correspondences) to dg categories. We denote by 
$$
\xymatrix{
\cS:Corr_k\ar[r] &  dgCat_k
}
$$ 
the underlying 1-categorical sheaf theory, i.e. the symmetric monoidal functor on $(\infty,1)$-categories obtained by forgetting noninvertible morphisms.
\end{defn}

Applying a sheaf theory to the geometric 
descriptions of traces of correspondences,
 one immediately deduces trace formulas for dg categories. We first spell out the consequences of the 1-categorial structure of a sheaf theory $\cS$, then the trace formulae arising from its 2-categorical enhancement $\uS$, and finally in Section~\ref{actual sheaf theory} explain how to use the results of~\cite{finiteness,GR} to deduce applications of this formalism.
 
\subsection{Dimensions and traces of sheaf categories: 1-categorical consequences}
The graph of a map of derived stacks $f:X\to Y$ provides a correspondence from $X$ to $Y$ and a correspondence from $Y$ to $X$. We denote the respective induced maps by 
$f_{*}:\cS(X)\to \cS(Y)$ and $f^!:\cS(Y)\to \cS(X)$.
The functoriality of $\cS$ concisely encodes base change for $f_{*}$ and $f^!$.
 For $\pi:X\to pt = \Spec k$, we denote by  $\omega_{X}=\pi^! k\in \cS(X)$ the 
$\cS$-analogue of the dualizing sheaf, and by
$\omega(X)=\pi_{*}\omega_{X} \in \cS(pt) = dgVect_k$ the $\cS$-analogue of  ``global volume forms". 
%We
%adopt traditional notations whenever possible, for example writing
%$\Gamma_\cS(X, \cF)=\pi_{\cS*}(\cF)$, for $\cF\in \cS(X)$.

Next we will record formal consequences of our prior calculations deduced from the fact
that a sheaf theory is symmetric monoidal.

\begin{prop}\label{basic integral transforms} Fix a sheaf theory $\cS:Corr_k\to dgCat_k$, and $X, Y\in Corr_k$.

(1) $\cS(X) \in dgCat_k$ is canonically self-dual, and for any $f:X\to Y$,  $f^!:\cS(Y) \to \cS(X)$ and  $f_{*}:\cS(X)\to \cS(Y)$ are canonically transposes of each other.

(2) $\cS(X)$ is canonically symmetric monoidal with tensor product
$$
\xymatrix{
\cF\ot^!\cG=\Delta^!(\pi_1^! \cF\ot \pi_2^! \cG) & \cF, \cG\in \cS(X)
}$$

(3)  For any $f:X\to Y$, the projection formula holds:
$$
\xymatrix{
f_*\cF\ot^! \cG \simeq f_*(\cF\ot^! f^!\cG) & \cF\in \cS(X), \cG\in \cS(Y)
}
$$

(4) There is a canonical equivalence of functors and integral kernels
$$
\Hom_{dgCat_k}(\cS(X),\cS(Y))\simeq \cS(X\times Y)
$$

(5) The functor $q_*p^!:\cS(X) \to \cS(Y)$ associated to a correspondence 
$$
\xymatrix{X&\ar[l]_-{p} Z \ar[r]^-{q} &Y}
$$
is represented by the integral kernel $(p\times q)_*\omega_Z \in \cS(X\times Y)$.
\end{prop}

\begin{proof}
(1) Follows immediately from 
Proposition \ref{basic stack dimensions}.

%
%The first assertion follows by applying the symmetric monoidal functor $\cS$ to the argument of Proposition \ref{basic stack dimensions}.
%Recall that the evaluation and coevaluation presenting the self-duality of $X$ are both
%given by $X$ itself as a correspondence between $\ast=\Spec k$ and
%$X\times X$ via the diagonal map. 
%Applying $\cS$ to this diagram and using the 
%symmetric monoidal structure we find
%\begin{equation}\label{selfdual S}
%\xymatrix{
%dgVect_k\ar[r]^-{\simeq} & \cS(\ast)\ar[r]^-{p^!}\ar@/_2pc/@{-}^-{\omega(\cL X)}[rrrrr] &\cS(X)\ar[r]^-{\Delta_*}& 
%\cS(X\times X)\ar[r]^-{\simeq}& \cS(X)\otimes \cS(X)
%\ar[r]^-{\Delta^!}&X\ar[r]^-{p_*}&\cS(\ast) 
%}
%\end{equation} 
%Here $\xymatrix{\cL X\ar[r]^-{\pi}& X}$ is the projection.

(2) Follows immediately from %The symmetric monoidal structure on $\cS(X)$ results immediately 
 the commutative algebra structure
on $X\in Corr_k$ (in fact commutative coalgebra structure on $X\in Stacks_k$) provided by the diagonal map.

(3)  Follows from base change for the diagram 
$$\xymatrix{
X\ar[d]_-{id\times f} \ar[r]^-{f} & Y\ar[d]^-{\Delta}\\
X\times Y\ar[r]^-{f\times \id} & Y\times Y
}
$$

(4) Since $\cS$ is monoidal, we have 
$$
\cS(X)\ot \cS(Y)\simeq \cS(X \times Y)
$$
The self-duality of $\cS(X)$ provides 
$$\Hom_{dgCat_k}(\cS(X),\cS(Y))\simeq \cS(X)^\vee\ot \cS(Y) \simeq \cS(X)\ot \cS(Y)
$$
By construction, the composite identification 
assigns the functor
$$
\xymatrix{
F_K(\cF)=\pi_{2_*}(\pi_1^!\cF\ot^! K)
&
K\in \cS(X\times Y)
}$$

(5) Follows from the projection formula: consider the diagram
$$\xymatrix{
& Z \ar[dl]_-{p} \ar[dr]^-{q} \ar[d]_-{\Pi}& \\
X & \ar[l]_-{\pi_1} X\times Y \ar[r]^-{\pi_2} & Y
}
$$
where $\Pi=p\times q$. Then we have
$$\xymatrix{
q_*p^!(-)\simeq \pi_{2*} \Pi_* \Pi^! \pi_1^!(-) \simeq
\pi_{2*} \Pi_* (\omega_Z \otimes^! \Pi^! \pi_1^!(-)) \simeq
\pi_{2*} (\Pi_* \omega_Z \otimes^! \pi_1^!(-))
}
$$
\end{proof}

\begin{prop}\label{trace of integral transform}
Fix a sheaf theory $\cS:Corr_k\to dgCat_k$.

(1) The $\cS$-dimension $\dim(\cS(X)) =HH_*(\cS(X))$ of any $X\in Corr_k$  is $S^1$-equivariantly equivalent with $\cS$-global volume forms on
the loop space 
$$
\dim(\cS(X))\simeq \omega(\cL X).
$$

 In particular,  for $G$ an affine algebraic group, characters
of $\cS$-valued $G$-representations are adjoint-equivariant $\cS$-global volume forms
 $$\dim(\cS(BG))\simeq \omega(G/G)
 $$

(2)  The $\cS$-trace of any endomorphism $Z\in Corr_k(X,X)$ is equivalent to $\cS$-global volume forms
on the restriction to the diagonal
$$
\Tr(\cS(Z)) \simeq \omega(Z|_{\Delta})
$$

 In particular, the $\cS$-trace of a self-map $f:X\to X$  is equivalent to $\cS$-global volume forms on the $f$-fixed point locus
$$
\Tr(f_{*})\simeq \omega(X^f)
$$
\end{prop}

\begin{proof}

(1) Follows immediately from Proposition~\ref{basic stack dimensions}(1). To spell this out, using the previous proposition and base change,
$\dim (\cS(X))$ results from applying the composition
$$
\xymatrix{
\pi_*\Delta^!\Delta_*\pi^! \simeq \pi_*p_{2*}p_1^!\pi^! \simeq \cL \pi_* \cL \pi^!:dgVect_k \ar[r] & dgVect_k
}
$$
to the unit $1_{dgVect_k} = k$. Here $\pi:X\to pt$ and $\cL \pi:\cL x\to pt$ are the maps to the point, and $p_1, p_2:\cL X \simeq X\times_{X\times X} X\to X$ are the two natural projections. Thus we find $\dim (\cS(X)) \simeq \cL \pi_* \cL \pi^!(k) \simeq \omega(\cL X)$.
Furthermore, 
 the $S^1$-equivariance results  from the one-dimensional cobordism hypothesis: the one-dimensional topological field theory defined by the dualizable object $\cS(X)\in dgCat_k$ factors through that defined by the dualizable object $X\in Corr_k$. Moreover, we identified the $S^1$-action on the dimension $\cL X$  with loop rotation.\footnote{One can also check directly that the cyclic structure on the cyclic bar construction of the dg category $\cS(X)$ is induced by
 the cyclic structure of the loop space $\cL X$ under the identification $\omega(\cL X)\simeq \dim(\cS(X))$.}
 
%  on $\cL X$, whence by functoriality we identify the $S^1$ action on $\dim(\cS(X))$ with the induced action from $\omega(\cL X)$. (When $\cS(X)=R\module$ for a dga $R$ one can see this explicitly from the cyclic structure on Hochschild homology). The identification $\dim(\cS(BG))=\omega(G/G)$ follows immediately. 

(2) Similarly follows immediately from Proposition~\ref{basic stack dimensions}(2).
%
%The trace of a correspondence $\xymatrix{X&\ar[l]_-{p_1} Z\ar[r]^-{p_2} & X}$ on $\cS(X)$ is calculated by the composition
%$$\xymatrix{ \cS(\ast) 
%\ar[r]^-{p^!} &\cS(X)\ar[r]^-{\Delta_*}& 
%\cS(X\times X)\ar[r]^-{p_1^!\otimes \id} & \cS(Z\times X) \ar[r]^-{p_{2*}\ot \id}& \cS(X\times X)\ar[r]^{\Delta^!} & \cS(X)\ar[r]^-{p_{X,*}}& \cS(\ast)}$$
%which we identify by base change with the composition $$\xymatrix{\cS(\ast) \ar[r]^-{p^!}& \cS(Z\times_{X\times X} X) \ar[r]^-{p_*} & 
%\cS(\ast),}$$
%i.e., with $\omega(Z|_{\Delta})$.  Finally specializing to the case of a self-map $Z=\Gamma_f$ 
%gives volume forms on the derived fixed point locus.
\end{proof}

%%%

\subsection{Integration formulas for traces: 2-categorical consequences}\label{2-cat}
Now we turn to the functoriality of dimensions and traces. For this we require the 2-categorical enhanced version $\uS$ of a sheaf theory, so as to take advantage of the resulting functorial adjunction
 $$\xymatrix{\cS(X) \ar@<+.5ex>[r]^{f_*} &
  \ar@<+.5ex>[l]^{f^!}\cS(Y)}$$ 
for $f:X\to Y$ proper. In particular, applying the functor $\uS$ on two-morphisms we find that

$\bullet$  A proper map $f:Z\to W$ of correspondences
$$
\xymatrix{
&\ar[dl]_-{p_Z} Z \ar[dd]^-f \ar[dr]^-{q_Z} & \\
X & & Y\\
& \ar[ul]^-{p_W} W\ar[ur]_-{q_W} & 
}
$$
induces a canonical integration morphism of integral transforms 
$$
\xymatrix{
\int_f: q_{Z*}p_Z^! \ar[r] & q_{W*} p_W^!
}$$

$\bullet$ In particular, when $X=Y=pt$, 
it induces a map of global volume forms
$$
\xymatrix{
\int_f: \omega(Z) \ar[r] & \omega(W)
}$$
 
$\bullet$ There is a canonical composition identity
 $$
 \xymatrix{
 \int_g\circ\int_{f}\simeq \int_{g\circ f}
 }
 $$
 
\begin{prop}\label{sheaf unit and counit} For $f:X\to Y$ proper, the unit and counit of the $(f_*,f^!)$ adjunction are given respectively by integration along
the proper maps of self-correspondences $\Delta_{/Y}:X\to X\times_Y X$ of $X$ and $f_{/Y}:X\to Y$ of $Y$:
$$
\xymatrix{
&\ar[dl]_-{} X \ar[dd]^-{\Delta_{f}} \ar[dr]^-{} & && &\ar[dl]_-{f} X \ar[dd]^-{f} \ar[dr]^-{f} & \\
X & & X  && Y & & Y\\
& \ar[ul]^-{p_1} X\times_Y X\ar[ur]_-{p_2} & && & \ar[ul]^-{} Y\ar[ur]_-{} & 
}
$$
\end{prop}

\begin{proof}
The assertion follows immediately from the geometric description of the unit and counit in the correspondence category, Proposition~\ref{geometric adjunction}, upon applying the functor $\uS$.
\end{proof}

\begin{prop}\label{dim via integration} 
For any proper map $f:X\to Y$, the induced map on dimensions 
$$
\xymatrix{
\dim(f_*):\dim(\cS(X))\ar[r] & \dim(\cS(Y))
}
$$
is identified ($S^1$-equivariantly) with integration along the loop map
$$
\xymatrix{
\dim(f_*)\simeq \int_{\cL f}:\omega(\cL X) \ar[r] & \omega(\cL Y)
}
$$
\end{prop}

\begin{proof}
%We need to calculate $\dim(f_*)$ (constructed in Definition \ref{functorial dim}) geometrically, following
%the construction of Proposition \ref{geometric dims} where the geometric dimension of $f$ was
%identified with the composition $$\xymatrix{\cL X\ar[r]& \cL_Y X \ar[r] & \cL Y.}$$
%The key is given by Propositions \ref{basic integral transforms} and \ref{integral transform prop}, which guarantee 
%that the sheaf theory constructions
%are prescribed by the geometry. 
%
According to Definition~\ref{functorial dim}, 
we must calculate the composition
$$
\xymatrix{
\dim(f_*):\Tr(\id_{\cS(X)})\ar[r]^-{\Tr(\eta_f)}& \Tr(f^!f_*) \ar[r]^-{\sim} & \Tr(f_*f^!) \ar[r]^-{\Tr(\epsilon_f)} & \Tr(\id_{\cS(Y)})
}
$$
The equivalence of the middle arrow is given by the canonical identifications 
$$
\Tr(f^!f_*)\simeq \omega((X\times_Y X) \times_{X \times X} X) \simeq \omega(X \times_{Y\times Y} Y) \simeq  \Tr(f_*f^!)
$$ 

%The construction is summarized in the following diagrams:
%
%
%
%$$
%\xymatrix{
% &                & \cS(X \times X) \ar@<-.5ex>[dd]_-{\id\ot f_*} \ar[rrdd]^-{p_{X*}\Delta^!}& &                                                            & &&&& \\
% & &  & &                                                                                                                                                  &&&&&\\
%\cS(\ast)\ar[rruu]^-{\Delta_*p_X^!}\ar[rr]^-{(f_*\ot \id)p_{X}^!}\ar[rrdd]_-{\Delta_*p_Y^!}&& 
%\cS(X\times Y)  \ar@<-.5ex>[uu]_-{\id\ot f^!} \ar[rr]^-{p_{X*}(\id\ot f^!)}\ar@<+.5ex>[dd]^-{f_*\ot \id}&&\cS(\ast)        & dgVect_k \ar[rrrr]^{\omega(\cL_Y X)} 
%\ar@/_4pc/@{-}[rrrr]^{\omega(\cL Y)} \ar@/^4pc/@{-}[rrrr]^{\omega(\cL X)}&&&& dgVect_k  \\
% & &  & &                                                                                                                                                     &&&&&\\
% &                & \cS(Y \times Y)\ar@<+.5ex>[uu]^-{f^!\ot\id} \ar[rruu]_-{p_{Y*}\Delta^!}& &                                                         &&&&&
%} 
%$$

%
%By Proposition \ref{basic integral transforms}, the composition $f^!f_*$ is realized by the
%correspondence $$\xymatrix{X&\ar[r]\ar[l] X\times_Y X& X},$$ and
%its trace as the restriction to the diagonal,
%$$\cL_Y X\simeq (X\times_Y X)\times_{X\times X} X.$$ Likewise the
%composition $f_*f^!$ is given by the correspondence $$\xymatrix{Y&\ar[r]\ar[l]
%  X& Y},$$ with trace its restriction to the diagonal,
%$$\cL_Y X\simeq X\times_{Y\times Y} Y.$$

By Proposition~\ref{sheaf unit and counit}, the unit $\eta_f:\id_{\cS(X)}\to f^! f_*$ is given by the integration morphism
$$
\xymatrix{\int_{\Delta_{f}}: \Delta_{f *}\omega_X \ar[r] & \omega_{X\times_Y X}
}
$$
and hence its trace $\Tr(\eta_f):\Tr(\id_{\cS(X)})\to \Tr(f^! f_*)$ is given by the induced integration map
$$
\xymatrix{
\int_{\Delta_f}:\omega(\cL X)\ar[r] & \omega((X\times_Y X) \times_{X \times X} X)
}
$$

Likewise, the counit $\epsilon_f:f_*f^!\to \id_{\cS(Y)}$ is given by by the integration morphism
$$
\xymatrix{
\int_f:f_*\omega_X\ar[r] & \omega_Y
}
$$  
and hence its trace $\Tr(\epsilon_f):\Tr(f_* f^!)\to Tr(\id_{\cS(Y)})$ is given by the induced integration map
$$
\xymatrix{
\int_{f}:\omega(X \times_{Y\times Y} Y )\ar[r] & \omega(\cL Y)
}
$$

Finally, by functoriality, their composition is given by the integration map 
$$
\xymatrix{
\int_{\cL f}:\omega(\cL X)\ar[r] & \omega(\cL Y)
}
$$
\end{proof}

Finally, we have the functoriality of traces in parallel with the previous theorem on the functoriality of  dimensions.
Let us recall the relevant setup. Consider a proper morphism $f:X\to Y$ and endomorphisms 
 $F_Z:X\to X$ and $F_W:Y\to Y$ in $Corr_k$ given by respective self-correspondences 
 $X\leftarrow Z \to X$ and $Y \leftarrow W \to Y$. 

  By an $f$-morphism from the pair $(X,F_Z)$ to the pair $(Y,F_W)$, we mean an
  identification 
  $$
  \xymatrix{
  s:Z  \ar[r]^-\sim &  X\times_Y W
  }$$ 
  of correspondences from $X$ to $Y$. This in turn
  induces an identification of what might be called relative traces
  $$
  \xymatrix{
  Z\times_{Y\times Y} Y\ar[r]^-\sim &  X\times_{Y\times Y} W
  }
  $$ 
  generalizing the relative loop space $\cL_X Y$ from the case of the identity correspondences $Z=X$, $W=Y$.
 We thus obtain a map of traces
$$
\xymatrix{
\tau(f, s):Z|_{\Delta_X}  =   Z\times_{X\times X} X \ar[r] &   Z \times_{Y\times Y} Y \ar[r]^-\sim &  X \times_{Y\times Y} W \ar[r] &  
  Y\times_{Y\times Y} W = W|_{\Delta_Y}
}
$$

\begin{prop} \label{trace via integration}
With the preceding setup, the trace map $\Tr(f_*,s):\Tr(F_{X*})\to\Tr(F_{Y*})$ is
canonically identified with the integration map 
$$
\xymatrix{
\int_{\tau(f, s)} : \omega( Z|_{\Delta_X}) \ar[r] & \omega(W|_{\Delta_Y} )
}$$
\end{prop}

\begin{proof}
The argument is parallel to the proof of Proposition \ref{dim via integration}. One calculates $\Tr(f_*,\alpha)$ from Definition~\ref{functorial trace} using Proposition~\ref{geometric traces}
and the compatibility of Proposition~\ref{basic integral transforms} and the integration morphism for integral transforms from Section~\ref{integration section}.
\end{proof}

\subsection{Ind-coherent sheaves and $\cD$-modules}\label{actual sheaf theory}

We now apply the results of Gaitsgory-Rozenblyum~\cite{GR} and Drinfeld-Gaitsgory~\cite{finiteness} establishing functoriality properties of categories of ind-coherent sheaves and $\cD$-modules.

We first state a fundamental 
result of Gaitsgory-Rozenblyum~\cite{GR}:

\medskip

\begin{thm}~\cite[Theorem III.3.5.4.3, III.3.6.3]{GR}  \label{GR sheaf theory}
There is a uniquely defined right-lax symmetric monoidal functor $\cQ^!$ from the $(\infty,2)$-category whose objects are {\em laft} (locally almost of finite type) prestacks, morphisms are correspondences with vertical arrow ind-inf-schematic, and 2-morphisms are ind-proper and ind-inf-schematic, to the $(\infty,2)$ category $\ul{dgCat}_k$ of $k$-linear presentable dg categories with continuous morphisms. The functor $\cQ^!$ is strictly symmetric monoidal on the full subcategory of {\em laft} ind-inf-schemes.
\end{thm}

The theorem encodes a tremendous amount of structure in great generality. Let us highlight some salient features useful in practice.
The theorem assigns a symmetric monoidal dg category $\cQ^!(X)$ to any stack satisfying a reasonable finite type assumption. 
The symmetric monoidal structure, the $!$-tensor product, is induced by $!$-pullback along diagonal maps. For an arbitrary morphism $p:X\to Y$ there is a continuous symmetric monoidal pullback functor $f^!:\cQ^!(Y)\to \cQ^!(X)$, while for $f$ schematic (or ind-schematic) there is a continuous pushforward $f_*:\cQ^!(X)\to \cQ^!(Y)$, which satisfies base change with respect to $!$-pullbacks. Moreover for $f$ proper (or ind-proper), $(f_*,f^!)$ form an adjoint pair. 

The theorem goes much further through the powerful formalism of {\em inf-schemes}: prototypical inf-schemes are quotients of schemes by infinitesimal equivalence relation. Thus one can treat on an equal footing ind-coherent sheaves that are equivariant for any formal groupoid. 
The most important example is the de Rham space $X_{dR}$ of a scheme, and one recovers $\cD$-modules on $X$ as ind-coherent sheaves on the de Rham functor of $X$, $\cD(X)=\cQ^!(X_{dR})$. Thus by first applying the functor $(-)_{dR}$ the theorem encodes the theory of $\cD$-modules, as a functor out of the correspondence 2-category of stacks (or {\em laft} prestacks) with pullbacks for arbitrary maps and pushforward for (ind-)schematic maps. 

\begin{corollary}
The functors $\cQ^!$ and $\cD$ define sheaf theories on {\em laft} ind-inf-schemes, i.e., define symmetric monoidal functors
$$\cQ^!,\cD: \uCorr(ind-inf-Sch_k)^{ind-prop}\longrightarrow \udg_k.$$
Thus the conclusions of Theorem~\ref{main} apply in this setting (in particular the Grothendieck-Riemann-Roch theorem for ind-proper maps
of ind-inf-schemes).
\end{corollary}

\subsubsection{The QCA setting.}\label{QCA section}
We are mostly interested in applications of sheaf theory on stacks, e.g. in an equivariant setting. That requires two features of the theories $\cQ^!$ and $\cD$ that are not encoded in Theorem~\ref{GR sheaf theory}.

Theorem~\ref{GR sheaf theory} produces in general a {\em right-lax} symmetric monoidal functor -- in other words, we have the natural map
$$\cQ^!(X)\ot\cQ^!(Y)\longrightarrow \cQ^!(X\times Y)$$ satisfying the expected coherences, but it is not an equivalence in general (though it is for schemes).

Also, while the theorem encodes arbitrary pullbacks, it does not encode a continuous pushforward functor $p_*:\cQ^!(X)\to \cQ^!(Y)$ for non-schematic morphisms (though a generalization to include QCA morphisms has been announced by the authors). This precludes an immediate application of our formalism to traces on stacks.

However, since the full structure of a sheaf theory is far stronger than is needed for the ``local" statements we discuss in this paper, we can get around this issue, by taking advantage of the following compilation of results of Drinfeld and Gaitsgory (specifically, see Section 3.6.1, Corollarys 3.7.14,  4.2.3, and 4.4.7, Proposition 4.4.11, Corollarys 8.3.4 and 8.4.3, Definition 9.3.2 and Proposition 9.3.12).

\begin{thm}~\cite{finiteness}\label{DG theorem}
\begin{enumerate}
\item For a QCA stack $X$, the categories $\cQ^!(X)$ and $\cD(X)$ are dualizable and canonically self-dual. 
\item The canonical functors define equivalences
$$\cQ^!(X)\otimes \cQ^!(Y)\simeq \cQ^!(X\times Y), \hskip.3in \cD(X)\otimes \cD(Y)\simeq \cD(X\times Y)$$
\item For a morphism $f:X\to Y$ of QCA stack, the ``renormalized pushforwards"\footnote{We note that for $\cQ^!$ the ``renormalized" pushforward is the standard pushforward functor, while for $\cD$ it differs for non-safe objects from the more familiar, but discontinuous, de Rham pushforward. } $$f_\bullet:\cQ^!(X)\to\cQ^!(Y), \hskip.3in f_\bullet:\cD(X)\to\cD(Y)$$ defined as the transpose of $f^!$ (by the self-duality of $(1)$) are continuous functors satisfying base-change and the projection formula with respect to pullback.
\end{enumerate}
\end{thm}

Let us now briefly indicate how the results of the previous two sections 
carry over to QCA stacks in the absence of a fully fledged sheaf theory, where we use the renormalized pushforward functors $f_\bullet$ provided by Theorem~\ref{DG theorem} to carry out non-representable pushforwards. In particular, for $X$ a general QCA stack, so that $pi_X:X\to pt$ is not representable, this means that the notation $\omega(X)$ has to be taken in a renormalized fashion,
$$\omega(X):=\pi_{X,\bullet}\pi_X^! k = \pi_{X,\bullet}\omega_X.$$
For the theory of ind-coherent sheaves this produces the usual notion of derived global sections of the dualizing complex, but for the theory of $\cD$-modules this will differ in general from the nonrenormalized version, namely Borel-Moore chains on $X$:
$$\omega(X)_{non-renorm}=R\Gamma_{dR}(\omega_X)=C_*^{BM}(X).$$

In Proposition~\ref{basic integral transforms}, the self-duality in assertion (1) for QCA stacks is the content of Theorem~\ref{DG theorem}(1), and $f_\bullet$ is defined so as to make it the transpose of $f^!$. Assertion (2) follows from the sheaf theory construction (i.e. is independent of non-representable morphisms), the projection formula is asserted in item (3) of the theorem and the last two assertions are deduced from the first three. Proposition~\ref{trace of integral transform} is also deduced directly from Proposition~\ref{basic integral transforms}.

The general trace construction $\int_f$ discussed in Section~\ref{2-cat} depends only on the functoriality of proper adjunction (which is part of the~\cite{GR} formalism) and the definition of pullback and (renormalized) pushforward. The identities needed to verify Propositions~\ref{sheaf unit and counit}, ~\ref{dim via integration} and~\ref{trace via integration} only depend on base change, which is guaranteed by Theorem~\ref{DG theorem}.
We therefore get as a payoff that the conclusions of Theorem~\ref{main} hold for QCA stacks, in particular:

\begin{thm} Let $\cS=\cQ^!$ or $\cS=\cD$ denote either ind-coherent sheaves or 
$\D$-modules. Let $f:X\to Y$ denote a proper morphism of QCA stacks.

$\bullet$  For 
any compact object $M\in \cS(X)$ (coherent sheaf or safe coherent $\cD$-module) with character $[M]\in HH_*(\cS(X))\simeq \omega(\cL X)$, 
there is a canonical identification
$$[f_*\cM]\simeq\int_{\cL f}[M]\in HH_*(\cS(Y))\simeq\omega(\cL Y)$$

\medskip

$\bullet$ Assume $Y=BG$ for an affine group and $X=Z/G$ for $Z$ a proper QCA stack.
Then for any compact object $M\in \cS(Z/G)$ ($G$-equivariant coherent sheaf or safely equivariant coherent $\cD$-module on $X$), and element $g\in G$,
there is a canonical identification
$$[f_*M]|_g \simeq \int_{\cL f} [M]|_{X^g}
$$ 

\medskip

$\bullet$ For a map $f:(X,Z)\to (Y,W)$ of QCA stacks with self-correspondences, the induced map $Tr(\cS(Z))\to Tr(\cS(W))$ is given 
by integration along fixed points $Z|_{\Delta_X}\to W_{\Delta_Y}$.

\end{thm}

%%%%%%%%%%%%%%%%%%%%%%%%%%%%%%%%%%%%%%%%%%%
%%%%%%%%%%%%%%%%%%%%%%%%%%%%%%%%%%%%%%%%%%%
%%%%%%%%%%%%%%%%%%%%%%%%%%%%%%%%%%%%%%%%%%%


\begin{thebibliography}{99}
%\bibitem[AJL]{AJL} L. Alonso Tarrio, A. Jeremias Lopez and J. Lipman, 
%    Bivariance, Grothendieck duality and Hochschild homology I. arXiv:1005.4328 
   
\bibitem[Ba]{barwick} C. Barwick, On the Q-construction for exact $\infty$-categories. e-print arXiv:1301.4725.

\bibitem[BFN]{BFN} D. Ben-Zvi, J. Francis and D. Nadler, Integral transforms
and Drinfeld centers in derived algebraic geometry. Preprint arXiv:0805.0157.
{\em Jour. Amer. Math. Soc.} {\bf 23} (2010), 909--966.


\bibitem[BN09]{character} D. Ben-Zvi and D. Nadler, The character theory
  of a complex group.  e-print arXiv:math/0904.1247

\bibitem[BN10b]{conns} D. Ben-Zvi and D. Nadler, Loop Spaces and
  Connections.  arXiv:1002.3636.  J. Topol. 5 (2012), no. 2, 377--430. 

\bibitem[BN10b]{reps} D. Ben-Zvi and D. Nadler, Loop Spaces and
  Representations.  arXiv:1004.5120. Duke Math. J. 162 (2013), no. 9, 1587--1619.

\bibitem[BN13]{secondary} D. Ben-Zvi and D. Nadler, Secondary Traces. e-print arXiv:1305.7177.

\bibitem[Cal1]{Cal1} A.~Caldararu and S. Willerton, The Mukai pairing. I. A categorical approach. arXiv:math/0308079. New York J. Math. 16 (2010), 61--98.  
 
\bibitem[Cal2]{Cal2} A.~Caldararu, The Mukai pairing, II:
The Hochschild-Kostant-Rosenberg
isomorphism,  Adv. Math.  194  (2005),  no. 1, 34–66, arXiv:math/0308080v3.

\bibitem[CT]{cisinskitabuada} D.-C. Cisinski and G. Tabuada,
 Lefschetz and Hirzebruch-Riemann-Roch formulas via noncommutative motives.
   arXiv:1111.0257.  J. Noncommut. Geom. 8 (2014), no. 4, 1171--1190.
   
\bibitem[DP]{doldpuppe}  A. Dold and D. Puppe, Duality, trace, and transfer. 
Proceedings of the International Conference on Geometric Topology (Warsaw, 1978), pp. 81--102, PWN, Warsaw, 1980.
   
\bibitem[DG]{finiteness} V. Drinfeld and D. Gaitsgory, On some finiteness questions for algebraic stacks. arXiv:1108.5351.  Geom. Funct. Anal. 23 (2013), no. 1, 149--294. 

\bibitem[EKMM]{EKMM} A. Elmendorf, I. Kriz, M. Mandell and J.P. May,
Rings, modules, and algebras in stable homotopy theory. With an appendix by M. Cole.
Mathematical Surveys and Monographs, 47.
American Mathematical Society, Providence, RI, 1997.


\bibitem[FG]{koszul} J. Francis and D. Gaitsgory, Chiral Koszul
  Duality. Selecta Math. (N.S.) 18 (2012), no. 1, 27--87. 

\bibitem[FHLT]{FHLT} D. Freed, M. Hopkins, J. Lurie and C. Teleman,
  Topological Quantum Field Theories from Compact Lie
  Groups. arXiv:0905.0731. A celebration of the mathematical legacy of Raoul Bott, 367--403, CRM Proc. Lecture Notes, 50, Amer. Math. Soc., Providence, RI, 2010.
  
\bibitem[FN]{FN} E. Frenkel and Ng\^o B.-C.,  Geometrization of trace formulas. Bull. Math. Sci. 1 (2011), no. 1, 129-199.

\bibitem[G1]{indcoh} D. Gaitsgory, Ind-coherent
  sheaves. arXiv:math/1105.4857  Mosc. Math. J. 13 (2013), no. 3, 399--528, 553.



\bibitem[G2]{1affine} D. Gaitsgory, Sheaves of categories and the notion of 1-affineness. Stacks and categories in geometry, topology, and algebra, 127--225, Contemp. Math., 643, Amer. Math. Soc., Providence, RI, 2015.

\bibitem[GR1]{crystals} D. Gaitsgory and N. Rozenblyum, Crystals and $\D$-modules. Pure Appl. Math. Q. 10 (2014), no. 1, 57--154.


\bibitem[GR2]{GR} D. Gaitsgory and N. Rozenblyum, A Study in Derived Algebraic Geometry. Draft available at http://math.harvard.edu/\~{}gaitsgde/GL/. Mathematical Surveys and Monographs 221, American Mathematical Society, 2017.

\bibitem[H1]{haugsengspan} R. Haugseng, Iterated spans and "classical" topological field theories. arXiv:1409.0837.

\bibitem[H2]{haugsengEn} R. Haugseng, The higher Morita category of $E_n$-algebras. e-print  arXiv:1412.8459. Geom. Topol. 21 (2017), no. 3, 1631--1730.


\bibitem[HSS]{HSS} M. Hoyois, S. Scherotzke and N. Sibilla, 
Higher traces, noncommutative motives, and the categorified Chern character. e-print  arXiv:1511.03589.
Adv. Math. 309 (2017), 97?154. 

\bibitem[Jo]{joshua}  R. Joshua, Riemann-Roch for algebraic stacks. I. Compositio Math. 136 (2003), no. 2, 117--169.

\bibitem[J]{joyce} D. Joyce,  An introduction to d-manifolds and derived differential geometry. Moduli spaces, 230--281, 
London Math. Soc. Lecture Note Ser., 411, Cambridge Univ. Press, Cambridge, 2014.    
    
\bibitem[Ke]{Ke} B.~Keller, On differential graded categories,
International Congress of Mathematicians. Vol. II, 151-190, Eur.
Math. Soc., Zurich, 2006.

\bibitem[KP]{KP} G. Kondyrev and A. Prihodko, Categorical proof of Holomorphic Atiyah-Bott formula.
 e-print arXiv:1607.06345 


\bibitem[Lo]{Loday} J.-L. Loday, Cyclic homology.
Appendix E by Mar\'ia O. Ronco.
Second edition.
Chapter 13 by the author in collaboration with Teimuraz Pirashvili.
Grundlehren der Mathematischen Wissenschaften [Fundamental Principles
 of Mathematical Sciences], 301. Springer-Verlag, Berlin,  1998.

\bibitem[Lu]{lunts} V. Lunts, Lefschetz fixed point theorems for
  Fourier-Mukai functors and DG algebras. arXiv:1102.2884.  J. Algebra 356 (2012), 230--256.
   

\bibitem[L1]{topos} J. Lurie, Higher topos theory.
arXiv:math.CT/0608040.  Annals of Mathematics Studies, 170. Princeton University Press, Princeton, NJ, 2009

\bibitem[L2]{HA} J. Lurie, Higher Algebra. Available at  http://www.math.harvard.edu/\~{}lurie/  

\bibitem[L3]{TFT} J. Lurie, On the classification of topological
  field theories.  Available at http://www.math.harvard.edu/\~{}lurie/
  Current developments in mathematics, 2008, 129-280, Int. Press,
  Somerville, MA, 2009.

\bibitem[L4]{SAG} J. Lurie, Spectral Algebraic Geometry.
Preprint, available at http://www.math.harvard.edu/\~{}lurie/


\bibitem[Ma]{markarian} N. Markarian, The Atiyah class, Hochschild
  cohomology and the Riemann-Roch theorem. arXiv:math/0610553. J. Lond. Math. Soc. (2) 79 (2009), no. 1, 129--143. 

\bibitem[M]{may}  J. P. May, Picard groups, Grothendieck rings, and Burnside rings of categories. Adv. Math. 163 (2001), no. 1, 1--16. 

\bibitem[Pe]{Petit} F.~Petit, A Riemann-Roch theorem for DG algebras, arXiv:1004.0361. Bull. Soc. Math. France 141 (2013), no. 2, 197--223. 

\bibitem[Po]{polishchuk} A. Polishchuk, Lefschetz type formulas for
  dg-categories arXiv:1111.0728. Selecta Math. (N.S.) 20 (2014), no. 3, 885--928.

\bibitem[PS]{PS} K. Ponto and M. Shulman,
    Traces in symmetric monoidal categories. arXiv:1107.6032

\bibitem[P]{toly} A. Preygel, Thom-Sebastiani and duality for
  matrix factorizations. arXiv:math/1101.5834


\bibitem[Ram]{Ram} A.~C.~Ramadoss, The relative Riemann-Roch theorem from
Hochschild homology, New York J. Math. 14 (2008) 643-717.

\bibitem[Ram2]{Ram2} A.~C.~Ramadoss, The Mukai pairing and integral transforms in
Hochschild homology, Moscow Math. Journal, Vol 10, No 3, (2010), 629-645.


\bibitem[Sh]{shipley THH} B. Shipley, Symmetric spectra and topological Hochschild homology.  $K$-Theory  19  (2000),  no. 2, 155--183.

\bibitem[Shk]{Shk} D.~Shklyarov, Hirzebruch-Riemann-Roch theorem for
DG algebras, arXiv:0710.1937v3.  Proc. Lond. Math. Soc. (3) 106 (2013), no. 1, 1--32.

\bibitem[Th]{Thomason} R.W. Thomason,
Lefschetz-Riemann-Roch theorem and coherent trace formula.
Invent. Math. 85 (1986), no. 3, 515--€"543. 

\bibitem[To]{Toen dg} B. To\"en, The homotopy theory of dg
categories and derived Morita theory. arXiv:math.AG/0408337. Invent.
Math.  167  (2007),  no. 3, 615--667.


\bibitem[To2]{Toen} B. To\"en, Higher and Derived Stacks: a global overview. Algebraic geometry, Seattle 2005. Part 1, 435--487, Proc. Sympos. Pure Math., 80, Part 1, Amer. Math. Soc., Providence, RI, 2009.

\bibitem[TV]{TV} B. To\"en and G. Vezzosi, Caract\`eres de Chern, traces \'equivariantes et g\'eom\'etrie alg\'ebrique d\'eriv\'ee. Selecta Math. (N.S.) 21 (2015), no. 2, 449--554.







\end{thebibliography}
\end{document}